\documentclass[11pt,a4paper,reqno]{amsart}
\usepackage{amsfonts,amsthm,amsmath,amscd}
\usepackage{booktabs}
\usepackage[T1]{fontenc}
\usepackage[utf8]{inputenc}
\usepackage[mathscr]{eucal}
\usepackage{indentfirst}
\usepackage{tikz}
\usepackage{algorithm}
\usepackage{algpseudocode}
\usetikzlibrary{shapes.geometric, positioning, calc}
\usetikzlibrary{automata,arrows,positioning,calc}
\usetikzlibrary{positioning}
\usepackage{caption}
\usepackage{subcaption}
\usepackage[font=small,labelfont=bf]{caption}
\usepackage{graphicx,graphics,pict2e}
\usepackage{tkz-berge}
\usetikzlibrary{arrows.meta} 
\usetikzlibrary{decorations.pathmorphing,decorations.markings} 
\usepackage{float}
\usepackage{epic}
\usepackage{lipsum}
\usepackage{stmaryrd}
\usepackage{authblk}
\usepackage[symbol]{footmisc}
\usepackage[normalem]{ulem}
\numberwithin{equation}{section}
\usepackage[margin=2.9cm]{geometry}
\usepackage{epstopdf}
\RequirePackage{doi}
\usepackage{hyperref}
\allowdisplaybreaks
\usepackage{cite}
\usepackage{multirow}
\usepackage{multicol}
\usepackage{rotating}
\usepackage{verbatim,amssymb,amsfonts}
\usepackage{mathrsfs}
\usepackage{mathtools}

\usepackage{tikz-cd}
\usepackage{indentfirst}
\usepackage{slashed}
\usepackage{thmtools}
\usepackage{setspace}
\usepackage{enumerate}
\usepackage{accents}
\usepackage{pdflscape}

\usetikzlibrary{arrows.meta, bending}

\usepackage{appendix}
\theoremstyle{plain}
\newtheorem{theorem}{Theorem}[section]
\newtheorem{lemma}[theorem]{Lemma}
\newtheorem{corollary}[theorem]{Corollary}
\newtheorem{proposition}[theorem]{Proposition}

\theoremstyle{definition}
\newtheorem{definition}[theorem]{Definition}
\newtheorem{remark}[theorem]{Remark}
\newtheorem{example}[theorem]{Example}

\DeclarePairedDelimiterX{\inp}[2]{\langle}{\rangle}{#1, #2}

\newcommand{\cE}{{\mathcal E}}

\newcommand{\ba}{\begin{eqnarray}}
	\newcommand{\na}{\end{eqnarray}}
\newcommand{\ban}{\begin{eqnarray*}}
	\newcommand{\nan}{\end{eqnarray*}}

\renewcommand{\thefootnote}{\fnsymbol{footnote}}

\makeatletter
\g@addto@macro{\endabstract}{\@setabstract}
\newcommand{\authorfootnotes}{\renewcommand\thefootnote{\@fnsymbol\c@footnote}}%
\makeatother

\makeatletter
\@namedef{subjclassname@2020}{%
	\textup{2020} Mathematics Subject Classification}
\makeatother

\title[]{Equivariant Analytic Spectral Invariants of \\ A\c{c}{\i}kme\c{s}e Lifts of Graphs with Self-Loops}



\begin{document}
	
	\begin{center}
		\vspace{-1cm}
		\maketitle		
		\normalsize
		\authorfootnotes
		Johnny Lim 
		\par \medskip
		
		{\small School of Mathematical Sciences, Universiti Sains Malaysia, Penang, Malaysia}\par \bigskip
	\end{center}
	
	\address{School of Mathematical Sciences, Universiti Sains Malaysia, Penang, Malaysia
	}
	\email{johnny.lim@usm.my}

	\begin{abstract}
	We introduce and study equivariant analytic spectral invariants associated with the A\c{c}{\i}kme\c{s}e lift of a graph with self-loops $G_S$. The canonical $\mathbb{Z}_2$-action yields an orthogonal isotypic decomposition of the lifted Laplacian into the anti-symmetric part $\mathcal{L}(G_S)$ and the symmetric part $M_{sym}$, from which we prove that the spectrum of $\mathcal{L}(G_S)$ is exactly the even-indexed spectrum of the lifted Laplacian. We express twisted moments of these blocks as traces of the generalised twisted moment operator against the corresponding $\mathbb{Z}_2$-projections, and obtain a tight upper bound under certain restriction with a characterization of the equality case. We introduce the equivariant heat character and establish a trace-norm stability estimate. For the matrix element of the resolvent $(pI+\mathcal{L}(G_S))^{-1}$ associated with the loop vector, we derive a	Laplace-transform identity for the equivariant heat character, a determinant formula, a spectral representation, and a Laurent expansion. We further characterize the case where $S$ is a union of connected components through resolvent and equivariant heat-character identities, and obtain explicit equivariant heat characters for joins of full-loop graphs and the line graph of a full-loop connected regular graph. Finally, we introduce the regularized equivariant heat integral and derive spectral and trace formulas for it.			
	\end{abstract}

	\section{Introduction}
	\label{Sect.1}
	
	This research began in 2024 with the author's attempt to study a characterization of the spectrum of complete bipartite graphs with self-loops at arbitrary positions, as a generalisation of the special case in \cite[Theorem 2.4]{akbari2023selfloop}. It turned out to be unfruitful because the method employed in the special case does not extend in general, mainly due to the Abel-Ruffini Theorem on the unsolvability of the general quintic, see e.g. \cite[\S 15]{stewart2022galois}, \cite[\S 14.7]{dummit2004abstract}, for which the occurrence of a general quintic prevents one from expecting a formula by radicals without additional algebraic restrictions. Regardless of the associated matrix under study, the main crux of this problem remains the complexity of the number of loops $\sigma$ intertwined with other graph invariants in the corresponding characteristic polynomial. One possible way out is to remove $\sigma$ while retaining its ``effect'' in a reasonable setting. This leads to studying the established spectral relation between the Laplacians of a self-loop graph and its simple graph, given by A\c{c}{\i}kme\c{s}e in \cite{acikmese2015spectrum} via a certain lifting.

	Such a lifting, which we shall call the \textit{A\c{c}{\i}kme\c{s}e lift} henceforth, is a simple graph of odd order obtained from $G_S$ by taking two copies of the corresponding simple graph and adjoining a new middle vertex to the two copies of each looped vertex.; see Fig.~\ref{fig:Alift}. A\c{c}{\i}kme\c{s}e showed in  \cite[Theorem~1]{acikmese2015spectrum} that the Laplacian spectrum of $G_S$ is contained in that of its lift, but no further explicit relation between the two spectra was reported. By verification with small-order examples, one observes that, after ordering the eigenvalues monotonically, the spectrum of $\mathcal L(G_S)$ appears exactly as the even-indexed spectrum of the lifted Laplacian. This motivates the author to pursue the following questions:
	
	\begin{enumerate}[(a).]
		\item For any self-loop graph $G_S,$ is it always true that its Laplacian spectrum can be retrieved from the even-indexed Laplacian spectrum of its A\c{c}{\i}kme\c{s}e lift $\widetilde{G}?$ If true, how can this be proved algebraically?   
		\item Beyond the Laplacian spectrum itself, what other spectral quantities associated with self-loop graphs can be extended to the A\c{c}{\i}kme\c{s}e lift? In particular, can the twisted-moment framework motivated in \cite{lim2025closedwalk} be extended and formulated in this setting?
		\item Since the A\c{c}{\i}kme\c{s}e lift admits a group action and an involution, what kind of equivariant analytic spectral invariants can be associated with this action, and what information about $G_S$ or the loop set $S$ can they retain?
	\end{enumerate}
	

	The first question is approached through an involution of the A\c{c}{\i}kme\c{s}e lift which interchanges the two copies of the underlying vertex set while fixing the middle vertex. This induces a $\mathbb Z_2$-action on the space of real-valued vertex functions of the lift. In Theorem~\ref{orthogonalequiv}, we show that the corresponding symmetric and anti-symmetric isotypic subspaces are invariant under the lifted Laplacian, with the restrictions orthogonally equivalent to $M_{sym}$ and $\mathcal L(G_S)$, respectively, where $\mathcal L(G_S)$ is a principal submatrix of $M_{sym}$. Combining this decomposition with the interlacing theorem gives an affirmative answer to Question~(a): if $\lambda_1\geq\cdots\geq\lambda_N$ and $\mu_1\geq\cdots\geq\mu_{2N+1}$ are the eigenvalues of $\mathcal L(G_S)$ and $\mathcal L(\widetilde G)$, respectively, then Theorem~\ref{mainthm1} gives
	\[
	\lambda_i=\mu_{2i},\qquad 1\leq i\leq N.
	\]
	
	The same $\mathbb Z_2$-decomposition is then applied to spectral moments. After determining the first few Laplacian spectral moments of the lift (see Theorem \ref{thm:spectral_moments}), we consider twisted moments relative to a center $c$ and introduce the corresponding generalised $c$-twisted moment operator by spectral functional calculus. Theorem \ref{momentdiff1} identifies the twisted moments of $\mathcal L(G_S)$ and $M_{sym}$ with traces obtained from the anti-symmetric and symmetric $\mathbb{Z}_2$-projections, respectively. Their difference is therefore given by an equivariant trace involving the involution matrix. Using a trace inequality for products of symmetric matrices, we obtain bounds for this difference. A refinement is given in Theorem \ref{thm:tightbound1}, where a tight upper bound is obtained together with a characterization of the equality case. The complete self-loop graph is subsequently considered to illustrate the spectral formulae and the attainment of this bound. This provides an answer to Question (b).
	
	The canonical involution also allows one to associate an \textit{equivariant heat character} with the A\c{c}{\i}kme\c{s}e lift. For the trivial element $1 \in \mathbb Z_2$, this recovers the ordinary heat trace, whilst for its nontrivial element $g \in \mathbb{Z}_2$, Theorem \ref{thm:isotypic_heat_decomposition} gives 
	\[
	\chi_t^{\widetilde G}(g)=\mathrm{Tr}(e^{-tM_{sym}})-\mathrm{Tr}(e^{-t\mathcal L(G_S)}).
	\]
	Thus, the ordinary and nontrivial equivariant heat characters recover the heat traces of the two isotypic blocks separately. In Proposition \ref{prop:heat_character_lipschitz}, we further prove a stability estimate for the equivariant heat character with respect to the trace norm of the lifted Laplacian using Duhamel's principle. One observes that $pI+\mathcal L(G_S),$ for every $p>0,$ is invertible. This motivates us to consider the matrix element $R_{G_S}(p)$ of the resolvent operator $(pI+\mathcal L(G_S))^{-1}$ associated with the loop vector $s$. 
	In Theorem \ref{thm:equivariant_resolvent}, we establish an identity for the Laplace transform of the equivariant heat character in terms of logarithmic derivative involving this matrix element:
	\[
	\int_0^\infty e^{-pt}\chi_t^{\widetilde{G}}(g)\,dt = \frac{d}{dp}\log\left(p + 2\sigma - 2R_{G,S}(p)\right), \quad p>0.
	\]
	In Theorem \ref{thm:loop_resolvent_spectral_measure_and_moments}, we further give the spectral representation and Laurent expansion of $R_{G,S}(p)$, which relate the poles and coefficients of the resolvent to the spectral projections of the loop vector and to the corresponding loop moments.
	
	These results admit a special case when the loop set $S$ is a union of connected components of the underlying graph. In Theorem \ref{thm:character_special_case}, we show that this condition is equivalent to $\mathcal L(G_S)s=s$, to the resolvent identity $R_{G,S}(p)=\sigma/(p+1)$, and to the equivariant heat-character formula using inverse Laplace transform
	\[
	\chi_t^{\widetilde{G}}(g)=1+e^{-(2\sigma+1)t}-e^{-t}.
	\]
	In particular, if the underlying graph is connected, this corresponds to the full-loop case. We then obtain explicit equivariant heat characters for several full-loop graph constructions. These include joins of full-loop graphs in Proposition \ref{prop:character_join_full_loop} and, after first deriving the corresponding Laplacian characteristic polynomial, the line graph of a full-loop connected regular graph in Proposition \ref{prop:line_graph_full_loop_character}.
	
	Finally, for a pseudo-connected self-loop graph, the nontrivial equivariant heat character satisfies $\chi_t^{\widetilde G}(g)\to 1$ as $t\to\infty$. This leads us to consider the \textit{regularized equivariant heat integral} $\mathcal I_{eq}(\widetilde G).$
	In Theorem \ref{thm:spectral_form_Ieq}, we establish that it can be alternatively expressed as a difference of traces:
	\[
	\mathcal I_{eq}(\widetilde G) \overset{\text{def}}{=}\int_0^\infty(\chi_t^{\widetilde G}(g)-1)\,dt = \mathrm{Tr}(M_{sym}^{\dagger}) - \mathrm{Tr}(\mathcal{L}(G_S)^{-1})
	\]
	involving the Moore--Penrose inverse of $M_{sym}$ and the inverse of $\mathcal L(G_S).$
	We further derive an expression in terms of inverse powers of $\mathcal L(G_S)$ and obtain a strict negative bound for it. These provide some answers to Question (c). 
	
	This article is organized as follows. Section \ref{Sect.2} recalls the Laplacian of a graph with self-loops and the construction of the A\c{c}{\i}kme\c{s}e lift. In Section \ref{Sect.3}, we formulate the canonical $\mathbb Z_2$-action and establish the isotypic decomposition and the	spectral result. Section \ref{Sect.4} establishes Laplacian spectral moments, generalised twisted moments and their associated trace inequalities. Section \ref{Sect.5} introduces the equivariant heat character and develops its relation with the resolvent matrix element associated with the loop vector. The special full-loop case, joins and line graphs are then considered, followed by the regularized equivariant heat integral.

	\vspace{-1em}
	\section{Preliminaries and Definitions}
	\label{Sect.2}

	Let $G=(V,E^\circ)$ be a finite undirected simple (without self-loops and multiple edges) graph of order $n=|V|,$  where $E^\circ$ is the edge set. Throughout, $A(G)=[a_{ij}]$ will be the adjacency matrix of $G,$ where the entry is $a_{ij}=1$ if $v_i$ and $v_j$ are adjacent and $a_{ij}=0$ otherwise;  $D(G)$ denotes the diagonal degree matrix of $G,$ i.e., the entries are diagonally $d_{i}$ being the degree of $v_i$ and zero otherwise. Let $S \subseteq V.$ Form a \textit{self-loop} graph $G_S$ from $G$ by attaching a loop at every vertex in $S.$ We denote $\sigma:=|S|$ as the number of loops on $G_S$. 
	
	For more background on algebraic graph theory and notions not fully elaborated in this article, we refer the readers to  \cite{biggs1993algebraic, cvetkovic1995spectra, cvetkovic2010intro, BrouwerHaemers}; for self-loop graphs, see \cite{akbari2023selfloop, akbari2024line, gutman2021energy, jovanovic2023}; and for relevant spectral Laplacians, see \cite{chung1997spectral, anchan20232, kinkar2014laplacian} Our basic setting primarily follows \cite{acikmese2015spectrum} and we develop further from that.
	
	\begin{definition}\cite{acikmese2015spectrum,chung1997spectral}
		\label{LaplacianLoop}
		Let $G$ be a simple graph. The Laplacian\footnote{We adopt the convention unnormalized Laplacians.} matrix of $G$ is defined by $\mathcal{L}(G)=D(G)-A(G),$ i.e.,
		$\mathcal{L}_{ij}=-1,$ for $i \neq j,$ if $v_i$ and $v_j$ are adjacent; $\mathcal{L}_{ii}=d_{i}$ the vertex degree on the diagonal, and zero otherwise. Let $S\subseteq V(G).$ For a self-loop graph $G_S,$ its Laplacian matrix is defined by
		\begin{equation}\label{eq:LaplacianLoop1}
			\mathcal{L}(G_S)=\mathcal{L}(G)+\sum_{i\in S} e_i e_i^\top,
		\end{equation}
		where $e_i\in\mathbb{R}^n$ is the $i$-th standard basis vector.
	\end{definition}
	
	More precisely, let $C$ be a vertex-edge incidence matrix of $G_S$. Then, $C$ can be partitioned as
  	$C= \begin{bmatrix}
			E_o\\ \mathcal{S}
		\end{bmatrix},$ 
	where $E_o\in\mathbb{R}^{m\times n}$ corresponds to the ordinary edges in $E^\circ$, and $\mathcal{S}\in\mathbb{R}^{\sigma\times n}$ corresponds to the self-loops. Since each row of $\mathcal{S}$ describes a self-loop, there is exactly one nonzero entry equal to $1$ in that row. If the $k$-th self-loop is incident to vertex $i_k\in\{1,\dots,n\}$, then the $k$-th row of $\mathcal{S}$ is
	$e_{i_k}^\top.$
	Equivalently,
	\begin{equation}
		\mathcal{S}_{k,j}
		=
		\begin{cases}
			1,& j=i_k,\\
			0,& j\neq i_k.
		\end{cases}
	\end{equation}
	It follows that
	$	\mathcal{S}^\top \mathcal{S}=\sum_{i\in S} e_i e_i^\top. $
	Hence, \eqref{eq:LaplacianLoop1} may also be written as
	\begin{equation}\label{eq:LaplacianLoop2}
		\mathcal{L}(G_S)=E_o^\top E_o+\mathcal{S}^\top \mathcal{S}.
	\end{equation}

	\begin{definition}\cite{acikmese2015spectrum}
		\label{def:pseudo-connected}
		An undirected self-loop graph without multiple edges is said to be \emph{pseudo-connected} if every connected component contains at least one vertex with a self-loop.
	\end{definition}
	
	A\c{c}{\i}kme\c{s}e \cite{acikmese2015spectrum} proved that the Laplacian $\mathcal{L}(G_S)$ of a pseudo-connected graph $G_S$ is positive definite. In the following, we prove a stronger result by establishing the converse.
	
	\begin{lemma}
		\label{lem:positive-definite}
		The Laplacian $\mathcal{L}(G_S)$ is positive definite if and only if every connected component of $G_S$ contains at least one self-loop. 
	\end{lemma}
	
	\begin{proof}
		For any $x=(x_1,\dots,x_n)^\top\in\mathbb{R}^n$, since
		$x^\top \mathcal{L}(G)x	=	\sum_{(i,j)\in E^\circ}(x_i-x_j)^2$, by using \eqref{eq:LaplacianLoop1} we have
		\begin{equation}
		x^\top \mathcal{L}(G_S)x
		= x^\top \mathcal{L}(G)x+\sum_{i\in S}x_i^2  
		= \sum_{(i,j)\in E^\circ}(x_i-x_j)^2+\sum_{i\in S}x_i^2. \label{eq:quadratic-form}
		\end{equation}
		Hence, $\mathcal{L}(G_S)$ is always positive semidefinite.
		
		Assume first that every connected component of $G_S$ contains at least one self-loop. Let $x\in\mathbb{R}^n$ satisfy $x^\top \mathcal{L}(G_S)x=0$. Since each summand in \eqref{eq:quadratic-form} is nonnegative, every summand must vanish. Therefore, we have
		\begin{enumerate}[(i)]
			\item $x_i=x_j$ for all $(i,j)\in E^\circ$, so $x$ is constant on each connected component of $G$;
			\item $x_i=0$ for all $i\in S$.
		\end{enumerate}
		Since each connected component contains a looped vertex, the constant value on each	component must be $0$. Hence, $x=0$, and so $\mathcal{L}(G_S)$ is positive definite.
		
		Conversely, suppose that some connected component $C$ of $G_S$ contains no self-loop. Let $x=\mathbf{1}_C$ be the indicator vector of $C$. Then, $x\neq 0$, and $x_i=x_j$ for every edge $(i,j)\in E^\circ$. Moreover, since $C\cap S=\emptyset$, the loop term vanishes on $C$. Thus, \eqref{eq:quadratic-form} gives $x^\top \mathcal{L}(G_S)x=0$, so $\mathcal{L}(G_S)$ is not positive definite.
	\end{proof}
	
	Now, we state the definition of our main operation:
	
	\begin{definition}\cite{acikmese2015spectrum}
		\label{liftgraph1}
		Suppose that $G_S=(V,E)$ is a graph with self-loops of order $n$ and with $\sigma\ge 1$. Its \textit{A\c{c}{\i}kme\c{s}e lift graph} $\widetilde{G}=(\widetilde{V},\widetilde{E})$ is a simple graph with $2n+1$ vertices constructed as follows: for every vertex $i\in V$, there exist corresponding vertices $i$ and $i+n+1$ in $\widetilde{G}$, together with a single middle vertex $n+1$. The edge set $\widetilde{E}$ is defined by:
		\begin{enumerate}[(i)]
			\item if $(i,j)\in E^\circ$, then
			$
			(i,j)\in \widetilde{E}
			\;\text{and}\;
			(i+n+1,j+n+1)\in \widetilde{E};
			$
			\item if $(i,i)\in E$, then
			$
			(i,n+1)\in \widetilde{E}
			\;\text{and}\;
			(n+1,i+n+1)\in \widetilde{E}.
			$
		\end{enumerate}
		We refer to the induced subgraph on the vertices $\{1,\dots,n\}$ and $\{n+2,\dots,2n+1\}$ as the
		\emph{lower copy} and  \emph{upper copy} of $G$, respectively. 
	\end{definition}	
	
	\vspace{-1em}
	
	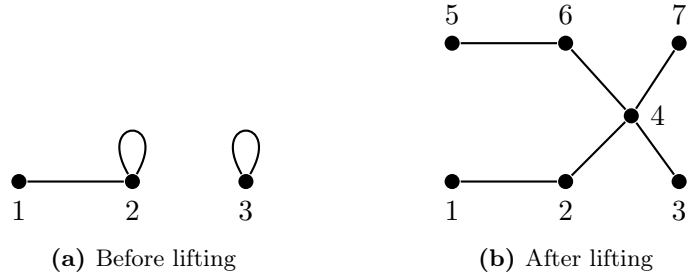
\begin{figure}[H]
		\begin{center}
			\begin{subfigure}{0.3 \textwidth}
							\centering
				\begin{tikzpicture}[node distance={15mm}, thick, 
					main/.style = {circle, fill=black, inner sep=2pt}] 
					
					\node[main, label=below:$1$] (1) {}; 
					\node[main, label=below:$2$] (2) [right of=1] {}; 
					\node[main, label=below:$3$] (3) [right of=2] {}; 
					
					\draw (1) -- (2); 
					\draw (2) to [out=60,in=120,looseness=20] (2); 
					\draw (3) to [out=60,in=120,looseness=20] (3); 
					
				\end{tikzpicture} 
				\subcaption{Before lifting}
			\end{subfigure}
			\hspace{2em}
			\begin{subfigure}{0.3\textwidth}
							\centering
				\begin{tikzpicture}[node distance={15mm}, thick, 
					main/.style = {circle, fill=black, inner sep=2pt}] 
					
					\node[main, label=below:$1$] (1) {}; 
					\node[main, label=below:$2$] (2) [right of=1] {}; 
					\node[main, label=below:$3$] (3) [right of=2] {};
					\node[main, label=right:$4$] (4) [above right=10mm of 2] {}; 
					\node[main, label=above:$5$] (5) [above=16mm of 1] {}; 
					\node[main, label=above:$6$] (6) [above=16mm of 2] {}; 
					\node[main, label=above:$7$] (7) [above=16mm of 3] {}; 
					
					\draw (1) -- (2); 
					\draw (5) -- (6); 
					\draw (2) -- (4); 
					\draw (3) -- (4); 
					\draw (6) -- (4); 
					\draw (7) -- (4); 
					
				\end{tikzpicture} 
				\subcaption{After lifting}
			\end{subfigure}
			\caption{An example and illustration of A\c{c}{\i}kme\c{s}e lift.}
			\label{fig:Alift}
		\end{center}
	\end{figure}

	Let $s:=\mathcal{S}^\top \mathbf{1}_\sigma\in\mathbb{R}^n.$
	Then, $s=(s_1,s_2,\ldots, s_n),$ where 
	\begin{equation} \label{eq:indicatorvec}
	s_i=
	\begin{cases}
		1,& i\in S,\\
		0,& i\notin S.
	\end{cases}
	\end{equation}
	For simplicity, we shall call such $s$ as the \textit{loop vector}.	For the lifted graph $\widetilde{G}$, choose an orientation of the ordinary edges in the lower and upper copies according to the incidence matrix $E_o$. For each self-loop at a vertex $i\in S$, the lift has two ordinary edges, $(i,n+1)$ and $(n+1,i+n+1).$ We orient the first from the lower copy to the middle vertex and the second from the middle vertex to the upper copy. In particular, an edge-vertex incidence matrix of $\widetilde{G}$ is
	\begin{equation}\label{eq:IncidenceLift}
		\widehat{E}
		=
		\begin{pmatrix}
			E_o & 0 & 0\\
			\mathcal{S} & -\mathbf{1}_\sigma & 0\\
			0 & \mathbf{1}_\sigma & -\mathcal{S}\\
			0 & 0 & E_o
		\end{pmatrix}.
	\end{equation}
	The choice of orientation does not affect the Laplacian, since changing the orientation	of any edge only multiplies the corresponding row of $\widehat{E}$ by $-1$, leaving $\widehat{E}^{\top}\widehat{E}$ unchanged. Using \eqref{eq:LaplacianLoop2}, the Laplacian\footnote{We remark that the Laplacian matrix in \cite[pp. 5]{acikmese2015spectrum} should be read with the signed incidence convention for ordinary edges in the lifted graph. Under this convention, the off-diagonal coupling blocks carry negative signs, and the middle diagonal entry is $2\sigma$, the degree of the middle vertex. As a result, the Laplacian in \eqref{eq:LaplacianLift1} is different than that of \cite[pp. 5]{acikmese2015spectrum}.} of the A\c{c}{\i}kme\c{s}e lift graph is
	\begin{equation}\label{eq:LaplacianLift1}
		\mathcal{L}(\widetilde{G})
		=
		\widehat{E}^{\top}\widehat{E}
		=
		\begin{pmatrix}
			\mathcal{L}(G_S) & -\mathcal{S}^\top \mathbf{1}_\sigma & 0\\
			-\mathbf{1}_\sigma^\top \mathcal{S} & 2\sigma & -\mathbf{1}_\sigma^\top \mathcal{S}\\
			0 & -\mathcal{S}^\top \mathbf{1}_\sigma & \mathcal{L}(G_S)
		\end{pmatrix}
		=
		\begin{pmatrix}
			\mathcal{L}(G_S) & -s & 0\\
			-s^\top & 2\sigma & -s^\top\\
			0 & -s & \mathcal{L}(G_S)
		\end{pmatrix}.
	\end{equation}	
	Indeed, the adjacency and degree matrices of $\widetilde{G}$ are, respectively,  
	\[
	A(\widetilde{G})=
	\begin{pmatrix}
		A(G) & s & 0\\
		s^\top & 0 & s^\top\\
		0 & s & A(G)
	\end{pmatrix}, \quad
	D(\widetilde{G})=
	\begin{pmatrix}
		D(G)+\mathrm{diag}(s) & 0 & 0\\
		0 & 2\sigma & 0\\
		0 & 0 & D(G)+\mathrm{diag}(s)
	\end{pmatrix}.
	\]
	Therefore, $\mathcal{L}(\widetilde{G})=D(\widetilde{G})-A(\widetilde{G})$, which aligns with \eqref{eq:LaplacianLift1}.
			
	Furthermore, it was established in  \cite[Theorem 1]{acikmese2015spectrum} that the inclusion relation
	\[
	\mathrm{Spec}(\mathcal{L}(G_S))\subseteq \mathrm{Spec}(\mathcal{L}(\widetilde{G}))\cap [\,0,\,2\Delta(G)+1\,]
	\]
	holds, where $\Delta(G)$ denotes the maximum degree of $G$.
	
	If $\mathcal{L}(G_S)v=\lambda v$, then by \eqref{eq:LaplacianLift1}, 
	\[
	\mathcal{L}(\widetilde{G})
	\begin{pmatrix}
		v\\
		0\\
		-v
	\end{pmatrix}
	=
	\begin{pmatrix}
		\mathcal{L}(G_S)v\\
		-s^\top v+s^\top v\\
		-\mathcal{L}(G_S)v
	\end{pmatrix}
	=
	\lambda
	\begin{pmatrix}
		v\\
		0\\
		-v
	\end{pmatrix},
	\]
	which shows that $\lambda \in \mathrm{Spec}(\mathcal{L}(\widetilde{G}))$.
	Moreover, since $\mathcal{S}^\top \mathcal{S}\le I_n$ (i.e., $I_n - \mathcal{S}^\top \mathcal{S}$ is positive semidefinite),  $\mathcal{L}(G_S)=\mathcal{L}(G)+\mathcal{S}^\top \mathcal{S}\le \mathcal{L}(G)+I_n.$
	Hence, we have
	\[
	\lambda_{\max}(\mathcal{L}(G_S))
	\le \lambda_{\max}(\mathcal{L}(G))+1 \le 2\Delta(G)+1.
	\]
	
	\section{Representation-Theoretic Formulation and $\mathbb{Z}_2$-Isotypic Decomposition}
	\label{Sect.3}
	
	Consider the abstract group $\mathbb{Z}_2=\{1,g\},$  where  $1$ is the identity element and $g$ is the generator satisfying $g^2=1.$ To analyze the spectral properties of the Laplacian $\mathcal{L}(\widetilde{G})$, we formally define an involution on the vertex set $\widetilde{V}$ and examine the resulting group action.
		
	\begin{definition}\label{involution1}
		Define a permutation mapping $\tau: \widetilde{V} \rightarrow \widetilde{V}$ by the following rules:
		\begin{align}
			\tau(i) = i+N+1, &\quad   
			\tau(i+N+1) = i  \qquad \text{for } 1 \le i \le N, \label{eq:tau_def} \\
			\tau(N+1) &= N+1. \nonumber
		\end{align}
	\end{definition}
	
	\begin{lemma}
	Let $\widetilde{G}= (\widetilde{V}, \widetilde{E})$ be the \textit{Aç{\i}kme\c{s}e lift graph} of a self-loop graph $G_S,$ where $N:= |V(G_S)|.$ Then, the permutation mapping $\tau$ defines a graph automorphism of $\widetilde{G}$ with an associated $\mathbb{Z}_2$ action on the vertex set $\widetilde{V}.$
	\end{lemma}
	
	\begin{proof}
		Let $(u, v) \in \widetilde{E}$. We must show that $(\tau(u), \tau(v)) \in \widetilde{E}$. We proceed by cases based on the construction of $\widetilde{E}$ in Definition \eqref{liftgraph1}:
		\begin{enumerate}
			\item Assume $(u, v)$ is an ordinary edge in the lower copy arising from an edge $(i, j) \in E^\circ$. Then, $\tau(i) = i+N+1$ and $\tau(j) = j+N+1$. By Definition \ref{liftgraph1}, $(i+N+1, j+N+1) \in \widetilde{E}$, thus preserving the adjacency. The converse holds identically for edges in the upper copy.
			\item Assume $(u,v)$ is the lifted edge $(i,N+1)\in \widetilde{E}$ arising from a self-loop $(i,i)\in E$ of $G_S$.
			Applying $\tau$ yields $(\tau(i), \tau(N+1)) = (i+N+1, N+1)$. Since $(N+1, i+N+1) \in \widetilde{E}$, the adjacency is preserved.
		\end{enumerate}
		Since $\tau$ preserves all adjacencies, it is a graph homomorphism. Furthermore, $\tau(\tau(v)) = v$ for all $v \in \widetilde{V}$, implying $\tau^2 = \text{id}.$ Thus, the map $\tau$ is bijective and $\tau^{-1}=\tau$. It follows that the implication
		\[
		(u,v)\in\widetilde{E}
		\Longrightarrow
		(\tau(u),\tau(v))\in\widetilde{E}
		\]
		also applies to $\tau^{-1}$ which gives the reverse implication. Hence, $\tau$ is a graph automorphism and the group $\langle \tau \rangle \cong \mathbb{Z}_2$ acts on the vertex set $\widetilde{V}$.
	\end{proof}
	
	The group action of $\mathbb{Z}_2$ on $\widetilde{V}$ naturally induces a linear representation on the real vector space vertex functions, $\mathbb{R}^{\widetilde{V}}$. Let $T \in \mathrm{GL}(\mathbb{R}^{\widetilde{V}})$ be the $(2N+1) \times (2N+1)$ permutation matrix associated with the action of $\tau:$ 
	\begin{equation}\label{eq:matrixT}
		T = 
		\begin{pmatrix} 
			0 & 0 & I_N \\ 
			0 & 1 & 0 \\ 
			I_N & 0 & 0 
		\end{pmatrix}.
	\end{equation} 
	More precisely, since $|\widetilde{V}|=2N+1,$ we may view $\mathbb{R}^{\widetilde{V}} \cong \mathbb{R}^{2N+1},$ where a vector $x =(x_u)_{u \in \widetilde{V}}$ represents a real-valued function defined on the vertices. Then, the action of $\tau$ on the vertices induces a linear action on the functions as: 
	\[
	(\tau \cdot x)_u:= x_{\tau^{-1}(u)}, \quad u \in \widetilde{V}.
	\]
	This action is linear and it is represented by the permutation matrix $T$ which acts on the standard basis vectors by $T e_u=e_{\tau(u)}.$ Thus, the linear action on any vector $x$ is simply the matrix multiplication $\tau \cdot x = Tx.$ In particular, $T$ is the matrix realization of a linear representation of the group $\mathbb{Z}_2:$ define $\rho: \mathbb{Z}_2 \to \mathrm{GL}(\mathbb{R}^{\widetilde{V}})$ by 
	\[
	\rho(1)=I_{2N+1} \quad \text{ and } \quad \rho(g)=T.
	\]
	Since $\tau^2=\mathrm{id}$, the associated permutation matrix satisfies $T^2=I_{2N+1}$. Therefore, $\rho(1)=I_{2N+1}$ and $\rho(g)=T$ respects the relation $g^2=1$ of $\mathbb{Z}_2$. Hence, $\rho$ is a representation.
	
	Recall that a representation $\rho : \mathrm{G} \to \mathrm{GL}(\mathcal{V})$ is said to be isotypic if it is a direct sum of isomorphic irreducible representation, and there is a canonical decomposition of $\mathcal{V}$ into the direct sum of isotypic representation, which we simply call isotypic decomposition, cf. \cite[\S 8]{serre1977linear}.
	
	\begin{proposition}\label{isotypicdecomp}
		The vector space $\mathbb{R}^{\widetilde{V}}$ admits the orthogonal $\mathbb{Z}_2$-isotypic decomposition:
		\begin{equation}\label{eq:decomp1}
			\mathbb{R}^{\widetilde{V}} = (\mathbb{R}^{\widetilde{V}})^+ \oplus (\mathbb{R}^{\widetilde{V}})^-
		\end{equation}
		where $(\mathbb{R}^{\widetilde{V}})^+ = \ker(T - I)$ and $(\mathbb{R}^{\widetilde{V}})^- = \ker(T + I)$ are the symmetric and anti-symmetric components, respectively.
	\end{proposition}
	
	\begin{proof}
		Since $\tau$ is an involution, the linear operator satisfies $T^2 = I$. The minimal polynomial of $T$ must therefore divide $x^2 - 1 = (x-1)(x+1)$. Moreover, since $\mathrm{char}(\mathbb{R}) =0 \neq 2,$ the roots are distinct. This guarantees that $T$ is diagonalizable and that $\mathrm{Spec}(T)\subseteq\{+1,-1\}$. In fact, since $\tau$ fixes the middle vertex and swaps $N$ pairs of vertices,
		\[
		\dim\ker(T-I)=N+1,\quad \dim\ker(T+I)=N.
		\]
		So, $\mathrm{Spec}(T)=\{+1,-1\}$ for $N\geq 1$.		
		Consequently, the vector space splits exactly into the direct sum of the respective eigenspaces. We define the $\mathbb{Z}_2$-equivariant orthogonal projection operators as:
		\begin{equation}
			P_{\pm} = \frac{1}{2}(I \pm T).
		\end{equation}
		It is straightforward to verify that 
		\[
		P_{\pm}^2 = P_{\pm}, \quad P_+ P_- = 0, \text{ and } P_+ + P_- = I.
		\]
		The images of these projectors yield the isotypic components: 
		\[
		\mathrm{Im}(P_+) = \ker(T-I) = (\mathbb{R}^{\widetilde{V}})^+, \quad \mathrm{Im}(P_-) = \ker(T+I) = (\mathbb{R}^{\widetilde{V}})^-.
		\] 
		
		Moreover, $T$ is symmetric, since its block form satisfies $T^\top=T$. Hence, the eigenspaces corresponding to the distinct eigenvalues $+1$ and $-1$ are orthogonal. Indeed, if $x\in\ker(T-I)$ and $y\in\ker(T+I)$, then
		\[
		\langle x,y\rangle
		=\langle Tx,y\rangle
		=\langle x,T^\top y\rangle
		=\langle x,Ty\rangle
		=-\langle x,y\rangle,
		\]
		showing that $\langle x,y\rangle=0$. 	\qedhere
	\end{proof}

	Using the permutation matrix $T$, we establish a characterization of the anti-symmetric isotypic subspace of $\mathbb{R}^{\widetilde{V}}.$
	\begin{lemma}
		By partitioning the vectors in $\mathbb{R}^{\widetilde{V}}$ into block form corresponding to the lower copy, the middle vertex, and the upper copy, the anti-symmetric isotypic subspace can be explicitly expressed as:
		\begin{equation}
			(\mathbb{R}^{\widetilde{V}})^- = \{(v, 0, -v) : v \in \mathbb{R}^N\}.
		\end{equation} 
	\end{lemma}
	
	\begin{proof}
		Denote $x \in \mathbb{R}^{\widetilde{V}}$ by $(x_L, \alpha, x_U)^T \in \mathbb{R}^N \oplus \mathbb{R} \oplus \mathbb{R}^N$, where $x_L$ corresponds to the lower copy of vertices, $\alpha$ to the middle vertex, and $x_U$ to the upper copy. The permutation matrix $T$ which represents the involution $\tau$ acts on $x$ by swapping the upper and lower copies while fixing the middle vertex:
		\begin{equation}\label{eq:Tswap1}
			T(x_L, \alpha, x_U)^\top = (x_U, \alpha, x_L)^\top.
		\end{equation} 
		Since $\mathrm{char}(\mathbb{R}) \neq 2,$ define 
		\[
		x_+=\frac{1}{2}(x+Tx), \quad x_-=\frac{1}{2}(x-Tx).
		\]
		By $T^2=I,$ we have
		\[
		Tx_- = T\left(\frac{1}{2}(x - Tx) \right) =  \frac{1}{2}(Tx - T^2x) = - \frac{1}{2}(x - Tx)=-x_-
		\] 
		or equivalently, $x_- \in \ker(T+I);$ similarly, $Tx_+=x_+$ or equivalently, $x_+ \in \ker(T-I).$ By definition, $x\in(\mathbb{R}^{\widetilde{V}})^-$ if and only if	$x\in\ker(T+I)$, or equivalently, $Tx=-x$.

		Combining with \eqref{eq:Tswap1} yields the necessary condition:
		\begin{equation}
			(x_U, \alpha, x_L)^\top = -(x_L, \alpha, x_U)^\top = (-x_L, -\alpha, -x_U)^\top.
		\end{equation}
		This strict equality forces $\alpha = -\alpha,$ so $\alpha = 0$, and $x_U = -x_L$. Let $v = x_L \in \mathbb{R}^N$, then any vector in $(\mathbb{R}^{\widetilde{V}})^-$ must be of the form $(v, 0, -v)^\top$. Conversely, if $x=(v,0,-v)^\top$, then $Tx=(-v,0,v)^\top=-x.$ Hence, $x\in\ker(T+I)=(\mathbb{R}^{\widetilde{V}})^-$.
	\end{proof}

	\begin{theorem}\label{orthogonalequiv}		
		Let  $\widetilde{G}=(\widetilde{V},\widetilde{E})$ be the A\c{c}{\i}kme\c{s}e lift of a self-loop graph $G_S$ of order $N$ and $\sigma=|S|$. Let $s$ be the loop vector. 		
	\begin{enumerate}[(i)]	
		\item By partitioning the vector space $\mathbb{R}^{\widetilde{V}}$  into $\mathbb{R}^N \oplus \mathbb{R} \oplus \mathbb{R}^N$, define $U_-$ and $U_+$ respectively by
		\begin{equation}\label{eq:U+-}
			U_- = 
			\begin{pmatrix} 
				\frac{1}{\sqrt{2}} I_N \\ 
				0_{1\times N} \\ 
				-\frac{1}{\sqrt{2}} I_N 
				\end{pmatrix} \in \mathbb{R}^{(2N+1) \times N}, \quad
			U_+ = 
			\begin{pmatrix}
				0_{N\times 1} & \frac{1}{\sqrt{2}}I_N\\
				1 & 0_{1\times N}\\
				0_{N\times 1} & \frac{1}{\sqrt{2}}I_N
			\end{pmatrix} \in \mathbb{R}^{(2N+1) \times (N+1)}.
		\end{equation}
		Then, $U = (U_-, U_+)$ is an orthogonal matrix satisfying $U^\top U = I_{2N+1}$. Furthermore, the column space of $U_-$ (resp. $U_+$) forms an orthonormal basis for $(\mathbb{R}^{\widetilde{V}})^-$ (resp. $(\mathbb{R}^{\widetilde{V}})^+$). 
		\item  Both isotypic components $(\mathbb{R}^{\widetilde{V}})^-$ and $(\mathbb{R}^{\widetilde{V}})^+$ are invariant subspaces of the Laplacian $\mathcal{L}(\widetilde{G})$. Moreover, on the anti-symmetric subspace, the restricted operator is orthogonally equivalent to the base Laplacian:
		\begin{equation}\label{eq:unitaryeq1}
			\mathcal{L}(\widetilde{G})\big|_{(\mathbb{R}^{\widetilde{V}})^-} \cong_{\mathrm{orth}} \mathcal{L}(G_S).
		\end{equation}
		On the symmetric subspace, the restricted operator is orthogonally equivalent to the $(N+1) \times (N+1)$ symmetric matrix $M_{sym}$:
		\begin{equation}\label{eq:unitaryeq2}
			\mathcal{L}(\widetilde{G})\big|_{(\mathbb{R}^{\widetilde{V}})^+} \cong_{\mathrm{orth}} M_{sym} := 
			\begin{pmatrix} 
				2\sigma & -\sqrt{2} \textbf{1}^\top_\sigma \mathcal{S} \\ 
				-\sqrt{2} \mathcal{S}^\top \textbf{1}_\sigma & \mathcal{L}(G_S) 
			\end{pmatrix},
		\end{equation}
		wherein $\mathcal{L}(G_S)$ forms an $N \times N$ principal submatrix.
		\end{enumerate}
	\end{theorem}
	
	\begin{proof}
		We first show (i). It is straightforward to verify that
		\[
		U_-^\top U_- = I_N, \quad U_+^\top U_+ = I_{N+1}, \quad  U_-^\top U_+ = 0.
		\]
		Thus, the matrix $U = (U_-, U_+)$ satisfies $U^\top U = I_{2N+1},$ proving it is an orthogonal matrix. Let $T$ be the permutation matrix associated with the involution $\tau$ on $\widetilde{V}$. By using the matrix form \eqref{eq:matrixT}, we verify the action of $T$ on the basis matrices.
		For $U_-$, we have
		\[
		T U_- = 
		\begin{pmatrix} 
			0   & 0 & I_N \\ 
			0   & 1 & 0 \\ 
			I_N & 0 & 0 
			\end{pmatrix} 
			\begin{pmatrix} 
			\frac{1}{\sqrt{2}} I_N \\ 
			0 \\ 
			-\frac{1}{\sqrt{2}} I_N 
			\end{pmatrix} 
			= 
			\begin{pmatrix} 
			-\frac{1}{\sqrt{2}} I_N \\ 
			0 \\ 
			\frac{1}{\sqrt{2}} I_N 
			\end{pmatrix} 
			= -U_-.
		\] Thus, every vector in the column space of $U_-$ lies in $\ker(T + I_{2N+1})$, which is exactly $(\mathbb{R}^{\widetilde{V}})^-$. Since $\text{dim}((\mathbb{R}^{\widetilde{V}})^-) = N$ and $U_-$ has $N$ linearly independent columns, it forms an orthonormal basis for $(\mathbb{R}^{\widetilde{V}})^-$. On the other hand, for $U_+,$ we obtain
		\[
		T U_+ 
		= \begin{pmatrix} 0 & 0 & I_N \\ 0 & 1 & 0 \\ I_N & 0 & 0 \end{pmatrix} \begin{pmatrix} 0 & \frac{1}{\sqrt{2}} I_N \\ 1 & 0 \\ 0 & \frac{1}{\sqrt{2}} I_N \end{pmatrix}
		= \begin{pmatrix} 0 & \frac{1}{\sqrt{2}} I_N \\ 1 & 0 \\ 0 & \frac{1}{\sqrt{2}} I_N \end{pmatrix} = U_+.
		\]
		Thus, every vector in the column space of $U_+$ lies in $\ker(T - I_{2N+1})=(\mathbb{R}^{\widetilde{V}})^+$. Using a similar argument as $U_-,$ we conclude that the column space of $U_+$ forms an exact orthonormal basis for $(\mathbb{R}^{\widetilde{V}})^+$.
				
		For (ii),  we verify that: with $\mathcal{L}(\widetilde{G})$ from \eqref{eq:LaplacianLift1},
		\begin{align*}
			\mathcal{L}(\widetilde{G})U 
			&= 
			\begin{pmatrix} 
				\mathcal{L}(G_S) & -s  & 0 \\ 
				-s^\top & 2\sigma & -s^\top \\ 
				0 & -s  & \mathcal{L}(G_S) 
			\end{pmatrix}
			\begin{pmatrix} 
				\frac{1}{\sqrt{2}} I_N & 0 & \frac{1}{\sqrt{2}} I_N \\
				 0 & 1 & 0\\ 
				 -\frac{1}{\sqrt{2}} I_N  &  0 & \frac{1}{\sqrt{2}} I_N
			\end{pmatrix} \\
			&= 
			\begin{pmatrix} 
				\frac{1}{\sqrt{2}} \mathcal{L}(G_S) & -s & \frac{1}{\sqrt{2}} \mathcal{L}(G_S)  \\ 
				0 & 2\sigma & -\sqrt{2} s^\top\\ 
				-\frac{1}{\sqrt{2}} \mathcal{L}(G_S)& -s& \frac{1}{\sqrt{2}} \mathcal{L}(G_S) 
			\end{pmatrix},
		\end{align*}	
		and thus,
			\begin{align*}
			U^\top \mathcal{L}(\widetilde{G})U 
			&= 
			\begin{pmatrix} 
				\frac{1}{\sqrt{2}} I_N & 0 & -\frac{1}{\sqrt{2}} I_N \\
				0 & 1 & 0\\ 
				\frac{1}{\sqrt{2}} I_N  &  0 & \frac{1}{\sqrt{2}} I_N
			\end{pmatrix}
				\begin{pmatrix} 
				\frac{1}{\sqrt{2}} \mathcal{L}(G_S) & -s & \frac{1}{\sqrt{2}} \mathcal{L}(G_S)  \\ 
				0 & 2\sigma & -\sqrt{2} s^\top\\ 
				-\frac{1}{\sqrt{2}} \mathcal{L}(G_S)& -s& \frac{1}{\sqrt{2}} \mathcal{L}(G_S) 
			\end{pmatrix}\\
			&= 
			\begin{pmatrix} 
				\mathcal{L}(G_S) & 0 & 0  \\ 
				0 & 2\sigma & -\sqrt{2} s^\top \\ 
				0 & -\sqrt{2}s& \mathcal{L}(G_S) 
			\end{pmatrix}
		\end{align*}	
	This matrix is block-diagonal with respect to the basis partition $U = (U_-, U_+)$, with the off-diagonal blocks connecting the two subspaces evaluating to zero. This implies that both isotypic components, $(\mathbb{R}^{\widetilde{V}})^-$ and $(\mathbb{R}^{\widetilde{V}})^+$, are invariant subspaces of $\mathcal{L}(\widetilde{G})$. Furthermore, the upper-left $N \times N$ block and middle-bottom right $(N+1) \times (N+1)$ block correspond to the action of $\mathcal{L}(\widetilde{G})$ restricted to the anti-symmetric and symmetric basis, respectively:
	\begin{align*}
		U^\top \mathcal{L}(\widetilde{G})U 
		&=	\begin{pmatrix} 
			U^\top_- \mathcal{L}(\widetilde{G})U_- & 0\\ 
			0& U^\top_+ \mathcal{L}(\widetilde{G})U_+
		    \end{pmatrix} \\
		&=  \begin{pmatrix} 
			\mathcal{L}(G_S) & 0\\ 
			0& M_{sym}
			\end{pmatrix} \cong_{\mathrm{orth}}
		\begin{pmatrix} 
			\mathcal{L}(\widetilde{G})\big|_{(\mathbb{R}^{\widetilde{V}})^-} & 0\\ 
			0& \mathcal{L}(\widetilde{G})\big|_{(\mathbb{R}^{\widetilde{V}})^+}
		\end{pmatrix} 
	\end{align*}
	which establish the orthogonal equivalences \eqref{eq:unitaryeq1} and \eqref{eq:unitaryeq2}.	
	\end{proof}

	\begin{remark}\label{REMARKorthogonalequiv}
	Following the orthogonal matrix $U= (U_-,U_+)$ in \eqref{eq:U+-}, it is a useful observation that under the same orthogonal decomposition, the matrix $T$ \eqref{eq:matrixT} satisfies
	\begin{equation}
	U^\top T U = 
	\begin{pmatrix} \label{eq:matrixT2}
		-I_N & 0 & 0 \\ 
		0 & 1 & 0 \\ 
		0 & 0 & I_N 
	\end{pmatrix} = 
	\begin{pmatrix} 
		-I_N & 0  \\ 
		0 & I_{N+1} 
	\end{pmatrix}, 
	\end{equation}
	which shows that $T$ is orthogonally similar to $\mathrm{diag}(-I_N, I_{N+1}).$  	
	\end{remark}	
	
	It is an immediate consequence from Theorem \ref{orthogonalequiv} to obtain the following result.
	
	\begin{theorem}
		\label{spectrumsum}
		Let $\widetilde{G}$ be the Aç{\i}kme\c{s}e lift of a self-loop graph $G_S.$ Let $\mathcal{L}(\widetilde{G})$  and $\mathcal{L}(G_S)$ be their Laplacians, respectively.  Let $(\mathbb{R}^{\widetilde{V}})^+$ and $(\mathbb{R}^{\widetilde{V}})^-$  be the symmetric and anti-symmetric isotypic subspaces of $\mathbb{R}^{\widetilde{V}},$ respectively. Then, with multiplicities, the spectrum of $\mathcal{L}(\widetilde{G})$ decomposes as		
		\begin{equation}\label{eq:spectrasplit}
			\mathrm{Spec}(\mathcal{L}(\widetilde{G}))
			=	\mathrm{Spec}(\mathcal{L}(G_S)) \uplus \mathrm{Spec}(M_{sym}),
		\end{equation}
		where $\uplus$ denotes multiset union.
	\end{theorem}
	
	We close this section by proving the spectral result on the recovering of spectrum of self-loop graphs from that of its A\c{c}{\i}kme\c{s}e lift, for which the latter provides a possible solution to the characterization problem aforementioned in Introduction.

	\begin{theorem}\label{mainthm1}
		Let $G_S$ be a self-loop graph of order $N$ with the eigenvalues of
		$\mathcal{L}(G_S)$ 
		\[
		\lambda_1\ge \lambda_2\ge \cdots \ge \lambda_N.
		\]
		Let $\widetilde{G}$ be the A\c{c}{\i}kme\c{s}e lift graph of $G_S$ with the eigenvalues of $\mathcal{L}(\widetilde{G})$
		\[\mu_1\ge \mu_2\ge \cdots \ge \mu_{2N+1}.\]
		Then, it holds that
		\[
		\lambda_i=\mu_{2i},\quad i=1,2,\dots,N.
		\] 
		Moreover, for fixed $c\in\mathbb{R}$, if $\displaystyle \delta_c(A)=\max_{\alpha\in\mathrm{Spec}(A)}|\alpha-c|$
		denotes the maximum spectral deviation of $A$ relative to $c$, then, 
		\[
		\delta_c(\mathcal{L}(\widetilde{G}))
		= \delta_c(M_{sym})	\ge	\delta_c(\mathcal{L}(G_S)).
		\]
	\end{theorem}

	\begin{proof}
		Let	$\mathrm{Spec}(\mathcal{L}(G_S))=\{\lambda_i\}_{i=1}^N,$	$\mathrm{Spec}(M_{sym})=\{\nu_j\}_{j=1}^{N+1},$ and $\mathrm{Spec}(\mathcal{L}(\widetilde{G})) = \left\{ \mu_k\right\}_{k=1}^{2N+1}$ be ordered non-increasingly. Since $\mathcal{L}(G_S)$ is an $N\times N$ principal submatrix of $M_{sym}$, the Interlacing Theorem \cite[Thm 1.3.15]{cvetkovic2010intro} gives 
		\begin{equation}\label{eq:specinterlaced}
		\nu_1 \ge \lambda_1 \ge \nu_2 \ge \lambda_2 \ge \dots \ge \nu_N \ge \lambda_N \ge \nu_{N+1}
		\end{equation}
		and this provides $2N+1$ eigenvalues already sorted in non-increasing order. By comparing with the spectrum of $\mathcal{L}(\widetilde{G}),$  we must have the even-indexed eigenvalues $\mu_{2i}=\lambda_i$ for $i=1,\dots,N,$ and the odd-indexed eigenvalues $\mu_{2j-1}=\nu_j$ for $j=1,\dots,N+1.$
		
		For the second claim, observe that from \eqref{eq:specinterlaced} that the smallest and largest eigenvalues of $\mathcal{L}(\widetilde{G})$ are $\nu_{N+1}$ and $\nu_1$, respectively. Hence, for every $c\in\mathbb{R}$,
		\[
		\delta_c(\mathcal{L}(\widetilde{G}))
		=\max\{|\nu_1-c|,|\nu_{N+1}-c|\}
		=\delta_c(M_{sym}).
		\]
		It follows that $\delta_c(M_{sym})\ge \delta_c(\mathcal{L}(G_S)).$
	\end{proof}

	\section{Spectral Moments and Twisted Inequalities}
	\label{Sect.4}

	In this section, we study the Laplacian analogue of spectral and twisted moments, as motivated in \cite{lim2025closedwalk}. By utilizing the orthogonal decomposition in Theorem~\ref{orthogonalequiv} we first determine the first three Laplacian spectral moments of the lift. Then, we study relative and generalised twisted moments and formulate their relation with the two $\mathbb Z_2$-isotypic components.
	
\begin{theorem}\label{thm:spectral_moments}
	The Laplacian spectral moments
	\[
	M_k(\mathcal{L}(\widetilde{G}))
	=
	\mathrm{Tr}(\mathcal{L}(\widetilde{G})^k)
	\]
	of the A\c{c}{\i}kme\c{s}e lift for the first three orders are as follows:
	\begin{align}
		M_1(\mathcal{L}(\widetilde{G}))
		&=
		2M_1(\mathcal{L}(G_S))+2\sigma,
		\label{eq:m1}\\
		M_2(\mathcal{L}(\widetilde{G}))
		&=
		2M_2(\mathcal{L}(G_S))+4\sigma^2+4\sigma,
		\label{eq:m2}\\
		M_3(\mathcal{L}(\widetilde{G}))
		&=
		2M_3(\mathcal{L}(G_S))+8\sigma^3+12\sigma^2
		+6s^\top\mathcal{L}(G_S)s,
		\label{eq:m3}
	\end{align}
	where $s$ is the loop vector.
\end{theorem}

\begin{proof}
	By Theorem \ref{orthogonalequiv}, we know that
	$\mathcal{L}(\widetilde{G}) \cong_{\text{orth}}
	\mathrm{diag}(\mathcal{L}(G_S), M_{sym}).	$
	Since the trace operator is invariant under cyclic permutations, we have
	\begin{equation} \label{eq:laplacianmoment1}
	M_k(\mathcal{L}(\widetilde{G}))=M_k(\mathcal{L}(G_S))+\mathrm{Tr}(M_{sym}^k).
	\end{equation}
	We proceed to compute $\mathrm{Tr}(M_{sym}^k)$ directly from \eqref{eq:unitaryeq2}, where
	\[
	M_{sym}
	=
	\begin{pmatrix}
		2\sigma & -\sqrt{2}s^\top\\
		-\sqrt{2}s & \mathcal{L}(G_S)
	\end{pmatrix}.
	\]
	For $k=1$, the trace is trivially
	\[
	\mathrm{Tr}(M^1_{sym})
	= 2\sigma+\mathrm{Tr}(\mathcal{L}(G_S))
	= 2\sigma+M_1(\mathcal{L}(G_S)),
	\]
	which immediately yields \eqref{eq:m1}.
	
	For $k=2$, by a direct computation, we have
	\[
	M_{sym}^2
	=
	\begin{pmatrix}
	4\sigma^2+2s^\top s	& -2\sigma\sqrt{2}s^\top-\sqrt{2}s^\top\mathcal{L}(G_S) \\
	-2\sigma\sqrt{2}s-\sqrt{2}\mathcal{L}(G_S)s	& 2ss^\top+\mathcal{L}(G_S)^2
	\end{pmatrix}.
	\]
	Since $s$ is the loop vector for $\sigma$ loops, the inner product gives
	$s^\top s=\sigma$. Thus,
	\begin{align}
		\mathrm{Tr}(M_{sym}^2)
		&= (4\sigma^2+2\sigma) + \mathrm{Tr}(2ss^\top+\mathcal{L}(G_S)^2) \nonumber\\
		&=	4\sigma^2+2\sigma + 2\mathrm{Tr}(ss^\top) + M_2(\mathcal{L}(G_S))  \nonumber\\
		&=	4\sigma^2+4\sigma+M_2(\mathcal{L}(G_S)), \label{eq:TrMsym2}
	\end{align}
	where the last equality follows from
	$\mathrm{Tr}(ss^\top)=\mathrm{Tr}(s^\top s)=\sigma$.
	Adding the remaining $M_2(\mathcal{L}(G_S))$ gives \eqref{eq:m2}.
	
	For $k=3$, we evaluate the trace of $M_{sym}^3=M_{sym}^2M_{sym}$.
	Computing the main diagonal blocks gives:
	\begin{align*}
		(M_{sym}^3)_{11}
		&=	2\sigma (4\sigma^2+2\sigma)	+ (-2\sigma\sqrt{2}s^\top-\sqrt{2}s^\top\mathcal{L}(G_S))(-\sqrt{2}s)\\
		&=	8\sigma^3+4\sigma^2	+ 4\sigma (s^\top s) + 2s^\top\mathcal{L}(G_S)s\\
		&=	8\sigma^3+8\sigma^2	+ 2s^\top\mathcal{L}(G_S)s.
	\end{align*}
	For the lower right block $(M_{sym}^3)_{22}$, we obtain:
	\begin{align*}
		(M_{sym}^3)_{22}
		&=(-2\sigma\sqrt{2}s - \sqrt{2}\mathcal{L}(G_S)s)(-\sqrt{2}s^\top) + (2ss^\top+\mathcal{L}(G_S)^2)\mathcal{L}(G_S)\\
		&=4\sigma ss^\top + 2\mathcal{L}(G_S)ss^\top+2ss^\top\mathcal{L}(G_S) + \mathcal{L}(G_S)^3.
	\end{align*}
	Applying the trace operator to $(M_{sym}^3)_{22}$, and using
	$\mathrm{Tr}(\mathcal{L}(G_S)ss^\top)=s^\top\mathcal{L}(G_S)s$, 
	\[
	\mathrm{Tr}((M_{sym}^3)_{22})
	= 4\sigma^2	+ 4s^\top\mathcal{L}(G_S)s + M_3(\mathcal{L}(G_S)).
	\]
	Summing the traces of the two blocks and adding the remaining
	$M_3(\mathcal{L}(G_S))$ gives
	\[
	M_3(\mathcal{L}(\widetilde{G}))
	= 2M_3(\mathcal{L}(G_S)) + 8\sigma^3 + 12\sigma^2 +	6s^\top\mathcal{L}(G_S)s,
	\]
	which yields \eqref{eq:m3}.
\end{proof}

\begin{theorem}\label{thm:zagreb_lift}
	Let $G_S$ be a self-loop graph of order $N$, where $G$ is the underlying simple graph with size $m$. Let $\sigma=\sum_{i=1}^N s_i$ where $s=(s_1,...,s_N)$ is the loop vector. Let $d_i=d_G(v_i).$ Then, the first Zagreb index of the A\c{c}{\i}kme\c{s}e lift is
	\[
	M_1(\widetilde{G})=2M_1(G)+4\sum_{i\in S}d_i+2\sigma+4\sigma^2.
	\]
	Moreover, $\mathrm{Tr}(M_{sym}^2)= \frac12 M_1(\widetilde{G})+2m+4\sigma+2\sigma^2.$
\end{theorem}

\begin{proof}
	In the lift, the lower copy of $v_i$ has degree $d_i+s_i$, and the upper copy has the same degree. The middle vertex is adjacent to the lower and upper copies of every	looped vertex, hence has degree $2\sigma$. It follows that
	\begin{align*}
		M_1(\widetilde{G})
		&= 2\sum_{i=1}^N(d_i+s_i)^2+(2\sigma)^2 \\
		&= 2\sum_{i=1}^Nd_i^2 + 4\sum_{i=1}^Nd_is_i +	2\sum_{i=1}^Ns_i + 4\sigma^2 \\
		&=2M_1(G) + 4\sum_{i\in S}d_i + 2\sigma+4\sigma^2.
	\end{align*}	
	Moreover, since
	$\mathcal{L}(G_S)=\mathcal{L}(G)+\mathrm{diag}(s), $
	we have
	\begin{align*}
		\mathrm{Tr}(\mathcal{L}(G_S)^2)
		&= \mathrm{Tr}(\mathcal{L}(G)^2) + 2\mathrm{Tr}(\mathcal{L}(G)\mathrm{diag}(s))	+ \mathrm{Tr}(\mathrm{diag}(s)^2) \\
		&= M_1(G)+2m + 2\sum_{i\in S}d_i + \sigma 	
	\end{align*}
	Therefore, by \eqref{eq:TrMsym2}, we obtain
	\begin{align*}
		\mathrm{Tr}(M_{sym}^2)
		&=	4\sigma^2+4\sigma+M_1(G)+2m+2\sum_{i\in S}d_i+\sigma \\
		&= 	\frac12M_1(\widetilde{G})+2m+4\sigma+2\sigma^2. \qedhere
	\end{align*} 
\end{proof}

\begin{example}
	Let	$G_S=(K_{a,b})_S$ be a complete bipartite self-loop graph with partite sets $A$ and $B$, where $|A|=a,|B|=b,$ for $a,b\ge 1.$ Suppose that $G_S$ has $\sigma_A$ looped vertices in $A$ and $\sigma_B$ looped vertices in $B$ so that $N=a+b$ and $\sigma=\sigma_A+\sigma_B.$	For the underlying simple graph $G=K_{a,b}$, we have $m=ab $ and it is readily computed that
	\[
	M_1(G) = \sum_{v\in V(G)}d_G(v)^2 = a b^2+b a^2 = ab(a+b), \quad \sum_{i\in S}d_i = b\sigma_A+a\sigma_B
	\]
	because every vertex in $A$ has degree $b$ and every vertex in $B$ has degree $a$. 
	
	By Theorem~\ref{thm:zagreb_lift}, the first Zagreb index of the
	A\c{c}{\i}kme\c{s}e lift is
	\[
	M_1(\widetilde{G})
	= 2ab(a+b) + 4(b\sigma_A+a\sigma_B) + 2\sigma + 4\sigma^2.
	\]
	Consequently,
	\begin{align*}
		\mathrm{Tr}(M_{sym}^2)
		&= \frac12 M_1(\widetilde{G}) + 2m + 4\sigma + 2\sigma^2\\
		&= ab(a+b) + 2(b\sigma_A + a\sigma_B) + \sigma + 2\sigma^2 + 2ab + 4\sigma + 2\sigma^2\\
		&= ab(a+b+2) + 2(b\sigma_A+a\sigma_B) + 5\sigma + 4\sigma^2.
	\end{align*}
\end{example}

We now extend the twisted-moment method to the lifted Laplacian setting. Recall from \cite{lim2025closedwalk} that, for a real symmetric matrix $B$ of order $r$ with eigenvalues $\lambda_1(B),\dots,\lambda_r(B)$, the $M_k$-twisted moment is defined by
\begin{equation}\label{eq:twistedmoment}
	\mathcal{M}_q^k(B)=\sum_{i=1}^r \left| \lambda_i(B)-\frac{M_k(B)}{r}\right|^q.
\end{equation}
In particular, when $k=1$, the center is the mean eigenvalue $\mathrm{Tr}(B)/r.$ We also write $\mathcal{M}_q(B):=\mathcal{M}^1_q(B).$

For the lifted Laplacian, we shall use the following relative notation. For a real symmetric matrix $B$ of order $r$ and a fixed $c\in\mathbb{R}$, define
\begin{equation}\label{eq:twistedmoment2}
	\mathcal{M}_q(B,c) = \sum_{i=1}^r |\lambda_i(B)-c|^q,\qquad q>0.
\end{equation}
For $q=0$, we set $\mathcal{M}_0(B,c)=r.$  Using this, we deduce the following relative form of the twisted-moment ratio inequality, which follows from the same Cauchy-Schwarz argument used in \cite[Theorem~3.4]{lim2025closedwalk}.

\begin{lemma}\label{lem:centered-ratio}
	Let $B$ be a real symmetric matrix of order $r$, and let $c\in\mathbb{R}$. 
	Suppose that $B\neq cI_r$. Then
	\[
	\frac{\mathcal{M}_1(B,c)}{\mathcal{M}_0(B,c)}
	\le
	\frac{\mathcal{M}_2(B,c)}{\mathcal{M}_1(B,c)}
	\le
	\frac{\mathcal{M}_3(B,c)}{\mathcal{M}_2(B,c)}
	\le
	\frac{\mathcal{M}_4(B,c)}{\mathcal{M}_3(B,c)}
	\le \cdots .
	\]
\end{lemma}

\begin{proof}
	Let $a_i:=|\lambda_i(B)-c|, i=1,\dots,r.$ Then, $a_i\ge 0$ and $\mathcal{M}_q(B,c)=\sum_{i=1}^r a_i^q.$	Since $B\neq cI_r$, at least one $a_i$ is nonzero, and hence $\mathcal{M}_q(B,c)>0$ for every $q>0$.
	
	By Cauchy-Schwarz inequality,
	\[
	\left(\sum_{i=1}^r a_i^q\right)^2
	=
	\left(\sum_{i=1}^r a_i^{\frac{q-1}{2}}a_i^{\frac{q+1}{2}}\right)^2
	\le
	\left(\sum_{i=1}^r a_i^{q-1}\right)
	\left(\sum_{i=1}^r a_i^{q+1}\right).
	\]
	Thus, $\mathcal{M}_q(B,c)^2	\le \mathcal{M}_{q-1}(B,c)\mathcal{M}_{q+1}(B,c),$
	and therefore,
	\[
	\frac{\mathcal{M}_q(B,c)}{\mathcal{M}_{q-1}(B,c)}
	\le
	\frac{\mathcal{M}_{q+1}(B,c)}{\mathcal{M}_q(B,c)}.
	\]
	Applying this successively gives the claimed ratio inequality.
\end{proof}

	In the following corollary, it is straightforward to deduce and obtain a generalization of the well-known classical bound
	\[
	\cE(G) \geq 2\sqrt{2}m\sqrt{\frac{m}{w^{cl}_4(G)}},
	\] 
	cf. \cite[\S 4, eq (11)]{MajKloGut2009} and references therein. The proof is omitted.
\begin{corollary}
	Let $B$ be a real symmetric matrix of order $r$, and let $c\in\mathbb{R}$. 
	Suppose that $B\neq cI_r$. Then, it holds that
	\begin{equation}\label{eq:B_bound}
		\mathcal{M}_1(B,c) \ge \sqrt{\frac{\mathcal{M}_2(B,c)^3}{\mathcal{M}_4(B,c)}}.
	\end{equation}
\end{corollary}

\begin{definition}
	For $q>0$ and $c\in\mathbb{R}$, define
	\begin{equation}\label{eq:twistmomentoperator}
		\Phi_q(\mathcal{L}(\widetilde{G}),c) = |\mathcal{L}(\widetilde{G})-cI_{2N+1}|^q,
	\end{equation}
	where the absolute value and the power are understood through the spectral functional calculus for real symmetric matrices; see, for example, \cite[\S 1]{higham2008functions}. We call $\Phi_q(\mathcal{L}(\widetilde{G}),c)$ the \textit{generalised $c$-twisted moment operator}
	for $\mathcal{L}(\widetilde{G})$.
\end{definition}

	How is $\Phi_q(\mathcal{L}(\widetilde{G}),c)$ related to $\mathcal{M}_q(\mathcal{L}(\widetilde{G}),c)$? For simplicity, let $X=\mathcal{L}(\widetilde{G})-cI_{2N+1}.$ Since $\mathcal{L}(\widetilde{G})$ is real symmetric, there is an orthogonal matrix $Q$ such that 
	\[
	X=Q \; \mathrm{diag}(\mu_1-c, \ldots, \mu_{2N+1}-c) \; Q^\top.
	\]
	Since $|X|=(X^2)^{1/2},$ using the same spectral decomposition and raising to the spectral power $q>0$, we must also have 
	\[
	|X|^q=Q \; \mathrm{diag}(|\mu_1-c|^q, \ldots, |\mu_{2N+1}-c|^q)\; Q^\top
	\]
	This means that the operator $\Phi_q(\mathcal{L}(\widetilde{G}),c)$ has eigenvalues $|\mu_1-c|^q, \ldots, |\mu_{2N+1}-c|^q.$
	Thus, for every $q>0$,
	\[
	\mathcal{M}_q(\mathcal{L}(\widetilde{G}),c)	= \sum^{2N+1}_{k=1} |\mu_k-c|^q = \mathrm{Tr}\big(\Phi_q(\mathcal{L}(\widetilde{G}),c)\big).
	\]

\begin{theorem}\label{momentdiff1}
	Let $q>0$ and $c\in\mathbb{R}$. Let
	$\rho:\mathbb{Z}_2\rightarrow \mathrm{GL}(\mathbb{R}^{\widetilde{V}})$ be the linear representation associated with the lift automorphism $\tau$, such that $\rho(g)=T$. Then, the $c$-twisted moments of $\mathcal{L}(G_S)$ and $M_{sym}$ satisfy
	\begin{align}
		\mathcal{M}_q(\mathcal{L}(G_S),c)
		&=\frac{1}{2}\mathrm{Tr}\left(\Phi_q(\mathcal{L}(\widetilde{G}),c)(I-T)\right),\\
		\mathcal{M}_q(M_{sym},c)
		&=\frac{1}{2}\mathrm{Tr}\left(
		\Phi_q(\mathcal{L}(\widetilde{G}),c)(I+T)
		\right).
	\end{align}
	Consequently,
	\begin{equation}\label{eq:momentdiff1}
		\mathcal{M}_q(M_{sym},c)-\mathcal{M}_q(\mathcal{L}(G_S),c)
		=\mathrm{Tr}\left(\Phi_q(\mathcal{L}(\widetilde{G}),c)T\right).
	\end{equation}
\end{theorem}

\begin{proof}
	Since $\tau \in \text{Aut}(\widetilde{G})$, the representation matrix $T$ commutes with the Laplacian, $[\mathcal{L}(\widetilde{G}), T] = 0$. Thus, $T$ commutes with the continuous operator $\Phi_q(\mathcal{L}(\widetilde{G}), c)$, ensuring they are simultaneously diagonalizable. Recall from Proposition \ref{isotypicdecomp} that the matrices $P_+ = \frac{1}{2}(I + T)$ and $P_- = \frac{1}{2}(I - T)$ are the orthogonal projection operators onto the symmetric subspace $(\mathbb{R}^{\widetilde{V}})^+$ and the anti-symmetric subspace $(\mathbb{R}^{\widetilde{V}})^-$, respectively. Let $\{w_k\}_{k=1}^{2N+1}$ be an orthonormal eigenbasis diagonalizing both $\mathcal{L}(\widetilde{G})$ and $T$, partitioned into $N+1$ symmetric eigenvectors $\{w_i^+\}$ and $N$ anti-symmetric eigenvectors $\{w_j^-\}$. 
	
	The projection $P_+$ acts as the identity on $\{w_i^+\}$ and vanishes on $\{w_j^-\}$. Evaluating the trace of the projected operator restricted to the symmetric subspace yields:
	\begin{align*}
		\frac{1}{2} \mathrm{Tr}\Big(\Phi_q(\mathcal{L}(\widetilde{G}), c)(I + T)\Big) 
		&= \mathrm{Tr}\Big(\Phi_q(\mathcal{L}(\widetilde{G}), c) P_+\Big) \\
		&= \sum_{i=1}^{N+1} (w_i^+)^\top \Phi_q(\mathcal{L}(\widetilde{G}), c) w_i^+ \\
		&= \sum_{\nu \in \mathrm{Spec}(M_{sym})} |\nu - c|^q \\
		&= \mathcal{M}_q(M_{sym}, c).
	\end{align*}
	Similarly, $P_-$ acts as the identity on $\{w_j^-\}$, we have 
	\begin{align*}
		\frac{1}{2} \mathrm{Tr}\Big(\Phi_q(\mathcal{L}(\widetilde{G}), c)(I - T)\Big) 
		&= \mathrm{Tr}\Big(\Phi_q(\mathcal{L}(\widetilde{G}), c) P_-\Big) \\
		&= \sum_{j=1}^N (w_j^-)^\top \Phi_q(\mathcal{L}(\widetilde{G}), c) w_j^- \\
		&= \sum_{\lambda \in \mathrm{Spec}(\mathcal{L}(G_S))} |\lambda - c|^q \\
		&= \mathcal{M}_q(\mathcal{L}(G_S), c).
	\end{align*}
	Consequently, the difference of these $c$-twisted moments is 
	\begin{align*}
		\mathcal{M}_q(M_{sym}, c) - \mathcal{M}_q(\mathcal{L}(G_S), c) 
		&= \frac{1}{2} \mathrm{Tr}\Big(\Phi_q (I + T)\Big) - \frac{1}{2} \mathrm{Tr}\Big(\Phi_q (I - T)\Big) \\
		&= \mathrm{Tr}\Big(\Phi_q(\mathcal{L}(\widetilde{G}), c) T\Big). \qedhere
	\end{align*}	
\end{proof}

\begin{theorem}\label{momentdiff2}
	Let $q>0$ and $c\in\mathbb{R}$. Let $e_1\ge e_2\ge \cdots \ge e_{2N+1}\ge 0$ denote the eigenvalues of $\Phi_q(\mathcal{L}(\widetilde{G}),c)$ in non-increasing order. Then,
	\begin{equation}\label{eq:momentdiff2}
		\sum_{i=N+1}^{2N+1}e_i-\sum_{i=1}^N e_i
		\le	\mathcal{M}_q(M_{sym},c)-\mathcal{M}_q(\mathcal{L}(G_S),c)
		\le	\sum_{i=1}^{N+1}e_i-\sum_{i=N+2}^{2N+1}e_i.
	\end{equation}
\end{theorem}

\begin{proof}
	By Theorem \ref{momentdiff1}, the middle term of \eqref{eq:momentdiff2} is $\mathrm{Tr}(\Phi_q(\mathcal{L}(\widetilde{G}),c)T).$ Let $e_i=$ $\lambda_i(\Phi_q(\mathcal{L}(\widetilde{G}),c)),$ and let $\lambda_i(T)$ denote the eigenvalues of $T$, both ordered non-increasingly. Since $T$ and $\Phi_q(\mathcal{L}(\widetilde{G}),c)$ are real symmetric matrices of order $2N+1$, there exists a trace inequality (cf. \cite[\S 9, Sect. H.1.g, Sect. H.1.h]{marshall2011inequalities},  \cite{theobald1975inequality}) such that
	\begin{equation}\label{eq:Neumanninequality1}
	\sum_{i=1}^{2N+1}
	\lambda_i(\Phi_q(\mathcal{L}(\widetilde{G}),c))
	\lambda_{2N+2-i}(T)
	\le
	\mathrm{Tr}\left(\Phi_q(\mathcal{L}(\widetilde{G}),c)T\right)
	\le
	\sum_{i=1}^{2N+1}
	\lambda_i(\Phi_q(\mathcal{L}(\widetilde{G}),c))
	\lambda_i(T).
	\end{equation}
	The spectrum of $T$ consists of $N+1$ eigenvalues equal to $+1$ and $N$ eigenvalues	equal to $-1$. Hence, in non-increasing order,
	\[
	\lambda_i(T)
   =\begin{cases}
		1,&1\le i\le N+1,\\
		-1,&N+2\le i\le 2N+1.
	\end{cases}
	\]
	The upper estimate therefore becomes $\sum_{i=1}^{N+1}e_i-\sum_{i=N+2}^{2N+1}e_i.$ For the lower estimate, the reversed sequence $\lambda_{2N+2-i}(T)$ assigns $-1$ to the first $N$ terms and $+1$ to the remaining	$N+1$ terms. Hence, the lower estimate becomes $\sum_{i=N+1}^{2N+1}e_i-\sum_{i=1}^{N}e_i. $
\end{proof}

	Whilst Theorem \ref{momentdiff2} gives us a bound for the difference where we utilize \eqref{eq:Neumanninequality1}, it is, however, a loose bound, after checking with some examples. 
	
	In the following, we attempt to derive a tighter upper bound and illustrate it with an example. 

\begin{theorem}\label{thm:tightbound1}
	Let $q>0$ and $c\leq 0$. Let $e_1 \ge e_2 \ge \dots \ge e_{2N+1} \ge 0$ denote the eigenvalues of $\Phi_q$ in \eqref{eq:twistmomentoperator}. Let $e_{2N+1}$ correspond to the constant vector $\mathbf{1} \in (\mathbb{R}^{\widetilde{V}})^+$. Let $\mathrm{Spec}^*(\Phi_q(M_{sym},c))$ denote the spectrum obtained from $\mathrm{Spec}(\Phi_q(M_{sym},c))$ by deleting precisely the single eigenvalue $e_{2N+1}$.	Then,
	\begin{equation}\label{eq:tightbound1}
		\mathcal{M}_q(M_{sym}, c) - \mathcal{M}_q(\mathcal{L}(G_S), c) \le e_{2N+1} + \sum_{i=1}^N e_i - \sum_{i=N+1}^{2N} e_i,
	\end{equation}
	where equality holds if and only if
	\begin{equation}\label{eq:tightbound2}
		\min_{\lambda \in \mathrm{Spec}^*(\Phi_q(M_{sym},c))} \lambda \ge \max_{\nu \in \mathrm{Spec}(\Phi_q(\mathcal{L}(G_S), c))} \nu.
	\end{equation}
\end{theorem}

\begin{proof}	
	Let $T$ be the involution matrix. Since $\mathbf{1} \in (\mathbb{R}^{\widetilde{V}})^+$, $T\mathbf{1} = \mathbf{1}$. Thus, $e_{2N+1}$ pairs with $+1 \in \mathrm{Spec}(T)$. Let $W = \mathrm{span}(\mathbf{1})$. Since $\mathbb{R}^{\widetilde{V}} = (\mathbb{R}^{\widetilde{V}})^+ \oplus (\mathbb{R}^{\widetilde{V}})^-$ and $W \subset (\mathbb{R}^{\widetilde{V}})^+$, its orthogonal complement is $W^\perp = ((\mathbb{R}^{\widetilde{V}})^+ \cap W^\perp) \oplus (\mathbb{R}^{\widetilde{V}})^-$, which has dimension $2N$. Let $A = \Phi_q|_{W^\perp}$ and $B = T|_{W^\perp}$. Apply Inequality \eqref{eq:Neumanninequality1} to $A$ and $B$ yields:
	\begin{equation}\label{eq:tightbound3}
	\mathrm{Tr}(AB)=\mathrm{Tr}(\Phi_q T) - e_{2N+1} \le \sum_{i=1}^N e_i - \sum_{i=N+1}^{2N} e_i.
	\end{equation}
	Apply \eqref{eq:momentdiff1} to \eqref{eq:tightbound3}, we obtain \eqref{eq:tightbound1}. From \cite{theobald1975inequality}, equality holds if and only if $A$ and $B$ admit a simultaneous ordered spectral decomposition.
	Since 
	$\mathrm{Spec}(B) = 
	\begin{bmatrix}
		+1 & -1 \\ 
		N & N
	\end{bmatrix},$ 
	its eigenspaces are $\mathbb{E}_B(+1) = (\mathbb{R}^{\widetilde{V}})^+ \cap W^\perp$ and $\mathbb{E}_B(-1) = (\mathbb{R}^{\widetilde{V}})^-$. Thus, the equality holds if and only if the eigenvectors corresponding to the $N$ largest eigenvalues of $A$ span $\mathbb{E}_B(+1)$, and those corresponding to the $N$ smallest eigenvalues of $A$ span $\mathbb{E}_B(-1)$. In particular, 
	\begin{equation}\label{eq:Thm4.9spec1}
	\min \mathrm{Spec}(A|_{\mathbb{E}_B(+1)})\ge \max \mathrm{Spec}(A|_{\mathbb{E}_B(-1)}).
	\end{equation}
	
	Conversely, since $A$ and $B$ commute, the eigenspaces $\mathbb{E}_B(+1)$ and $\mathbb{E}_B(-1)$ are $A$-invariant. Hence, if \eqref{eq:Thm4.9spec1} holds, their
	respective orthonormal eigenbases give a simultaneous ordered spectral decomposition of $A$ and $B$,. Thus, the equality follows. Now, note that
	\[
	\mathrm{Spec}(A|_{\mathbb{E}_B(+1)}) = \mathrm{Spec}^*(\Phi_q(M_{sym},c))
	\]
	because the symmetric block of $\Phi_q$ is $\Phi_q(M_{sym},c)$ and the restriction	to $W^\perp$ deletes precisely the single eigenvalue $e_{2N+1}$ corresponding to $W=\mathrm{span}(\mathbf{1})$. On the other hand, we have
	\[
	\mathrm{Spec}(A|_{\mathbb{E}_B(-1)}) = \mathrm{Spec}(\Phi_q(\mathcal{L}(G_S),c)).
	\]
	Hence, the equality holds if and only if
	\[
	\min_{\lambda\in \mathrm{Spec}^*(\Phi_q(M_{sym},c))}\lambda
	\ge
	\max_{\nu\in \mathrm{Spec}(\Phi_q(\mathcal{L}(G_S),c))}\nu.
	\]
	This proves \eqref{eq:tightbound2}.
\end{proof}

The upper bound in \ref{eq:tightbound1} is tight and is attained by a graph family as illustrated in Example \ref{Kn_Example_min_max_equality} below. Before that, let us prove a corollary.

\begin{corollary}
	Let $G_S$ be pseudo-connected, and let $c\le 0$ and $q>0$. Let $\mathrm{Spec}_0(M_{sym})$ denote the spectrum of $M_{sym}$ obtained by deleting its unique zero eigenvalue. 
	Then, the upper bound in \eqref{eq:tightbound1} is attained if and only if
	\[
	\min_{\alpha\in \mathrm{Spec}_0(M_{sym})}\alpha
	\ge \max_{\lambda\in \mathrm{Spec}(\mathcal{L}(G_S))}\lambda.
	\]
	In particular, if the smallest nonzero eigenvalue of $\mathcal{L}(\widetilde{G})$ belongs to the symmetric block and is less than $\max(\mathrm{Spec}(\mathcal{L}(G_S)))$, then the upper bound is not attained.
\end{corollary}

\begin{proof}
	Since $G_S$ is pseudo-connected, $\mathcal L(G_S)$ is positive definite by Lemma \ref{lem:positive-definite}. Moreover, the A\c{c}{\i}kme\c{s}e lift $\widetilde G$ is connected. Therefore, $\mathcal L(\widetilde G)$ has exactly one zero eigenvalue. Since $\mathcal L(G_S)$ has no zero eigenvalue, by Theorem \ref{spectrumsum},  $M_{sym}$ has exactly one zero eigenvalue. 
	
	Since $c\le 0$ and $q>0$, the function $f(x)=|x-c|^q$ is strictly increasing on $[0,\infty)$. Since both $\mathcal{L}(G_S)$ and	$M_{sym}$ are positive semidefinite, all their eigenvalues lie in $[0,\infty)$. By spectral functional calculus, we have
	\[
	\mathrm{Spec}(\Phi_q(M_{sym},c))
	=f(\mathrm{Spec}(M_{sym}))
	\quad \text{ and } \quad
	\mathrm{Spec}(\Phi_q(\mathcal{L}(G_S),c))
	= f(\mathrm{Spec}(\mathcal{L}(G_S))).
	\]
	Moreover, let $v_0$ denote the eigenvector of $M_{sym}$ corresponding to that zero eigenvalue. By spectral functional calculus,  $\Phi_q(M_{sym},c)v_0=f(0)v_0=|c|^qv_0.$ Thus, that zero eigenvalue of $M_{sym}$ corresponds to the eigenvalue $|c|^q$ of $\Phi_q(M_{sym},c)$.
	Since $\mathrm{Spec}_0(M_{sym})$ is the spectrum obtained from $\mathrm{Spec}(M_{sym})$ by deleting that zero eigenvalue, we have
	\begin{equation}\label{eq:funCalMsym1}
	\mathrm{Spec}^*(\Phi_q(M_{sym},c)) = f(\mathrm{Spec}_0(M_{sym})).
	\end{equation}
	
	Apply \eqref{eq:funCalMsym1} in \eqref{eq:tightbound2}, we get
	\[
	\min_{\alpha\in f(\mathrm{Spec}_0(M_{sym}))}\alpha
	\ge
	\max_{\beta\in f(\mathrm{Spec}(\mathcal{L}(G_S)))}\beta.
	\]
	Since $f$ is strictly increasing on $[0,\infty)$, this is equivalent to
	\[
	\min_{\alpha \in \mathrm{Spec}_0(M_{sym})}\alpha
	\ge
	\max_{\beta \in \mathrm{Spec}(\mathcal{L}(G_S))}\beta.
	\]
	The final assertion follows immediately that if the smallest nonzero eigenvalue in the symmetric block is less than $\max(\mathrm{Spec}(\mathcal{L}(G_S)))$, then the equality condition fails.
	\end{proof}

\subsection{Example: Complete self-loop graphs}

	In the following theorem, we utilize the approach of equitable partition to determine the respective spectrum, see \cite[\S 9]{godsil2001algebraic}  and \cite[\S 2.3]{BrouwerHaemers} for more details.

\begin{theorem}
	Let $G_S = (K_n)_S$ with $S \subset V(K_n)$ and $|S| = \sigma$, $0 < \sigma < n$. Then,
	\begin{equation}
		\mathrm{Spec}(\mathcal{L}(G_S)) 
		=\begin{bmatrix}
			n+1 & n &  \lambda_+ & \lambda_-\\ 
			\sigma-1 & n-\sigma-1 & 1 & 1
		\end{bmatrix}
	\end{equation}
	\begin{equation}
		\mathrm{Spec}(M_{sym}) 
		=\begin{bmatrix}
		n+1 & n & 0 & \eta_+ & \eta_-\\ 
		\sigma-1 & n-\sigma-1 & 1 & 1 & 1
		\end{bmatrix}		
	\end{equation}	
	where 
	\begin{align}
	\lambda_\pm &=  \frac{n+1 \pm \sqrt{(n-1)^2 + 4(n-\sigma)}}{2}, \label{eq:lambda+-}\\
	\eta_\pm &= \frac{(n+2\sigma+1) \pm \sqrt{(n+2\sigma+1)^2 - 4\sigma (2n+1)}}{2} \label{eq:eta+-}
	\end{align}
\end{theorem}

\begin{proof}
	Let $\Sigma=\mathrm{diag}(s_1,\ldots, s_n).$ Since $\mathcal{L}(K_n) = nI - J$, we have $\mathcal{L}(G_S) = nI - J + \Sigma$. In the case of looped vertices, first we consider the subspace
	\[
	X_S=\left\{x\in\mathbb R^n : x_i=0 \text{ for } i\notin S,\ \sum_{i\in S}x_i=0\right\}.
	\]
	For every $x\in X_S$, we have $\Sigma x=x$ and $Jx=\mathbf{0}$. Thus, 
	$
	\mathcal{L}(G_S)x = (nI - J + \Sigma)x = (n+1)x.
	$
	Since $\dim(X_S) = \sigma - 1$, this gives the eigenvalue $n+1$ with multiplicity $\sigma-1$.
	
	Similarly, in the case of unlooped vertices, consider 
	\[
	Y_S=\left\{y\in\mathbb R^n: y_i=0 \text{ for } i\in S,\ \sum_{i\notin S}y_i=0\right\}.
	\]
	For every $y\in Y_S$, we have $\Sigma y=\mathbf{0}$ and $Jy=\mathbf{0}$. Thus,
	$
	\mathcal{L}(G_S)y = ny. 
	$
	Since $\dim Y_S=n-\sigma-1$, this gives the eigenvalue $n$ with multiplicity $n-\sigma-1$.
	
	Partition $V$ into $\Pi = \{V_1, V_2\},$ where $V_1 = S$ and $V_2 = V \setminus S$.	Let $q_{ij}$ denote the sum of entries in $\mathcal{L}(G_S)$ from row $u \in V_i$ over columns $v \in V_j$. Then, for $u \in V_1$, 
	\begin{align*}
	q_{11} &= n(1) + (-1)(|V_1| - 1) = n - \sigma + 1 \\
	q_{12} &= -|V_2| = -(n - \sigma).
	\end{align*}
	As for $u \in V_2$, we can similarly obtain $q_{21} = -|V_1| = -\sigma$ and $q_{22} = n(1) - |V_2| + 0 = \sigma$. Since $q_{ij}$ is constant for all $u \in V_i$, $\Pi$ is an equitable partition. The quotient matrix of $\mathcal{L}(G_S)$ over $\Pi$ is:
	\begin{equation}\label{eq:quomat1}
	Q_- = \begin{pmatrix} n-\sigma+1 & -(n-\sigma) \\ -\sigma & \sigma \end{pmatrix}. 
	\end{equation}
	Since $\Pi$ is equitable, $\mathrm{Spec}(Q_-) \subseteq \mathrm{Spec}(\mathcal{L}(G_S))$. Since $\det(\xi I - Q_-) = \xi^2 - (n+1)\xi + \sigma$, its roots yield the remaining eigenvalues of $\mathcal{L}(G_S),$ cf. \cite[Theorem 9.3.3]{godsil2001algebraic}. Together with the eigenvalue $n+1$ of multiplicity $\sigma-1$ and the eigenvalue $n$ of multiplicity $n-\sigma-1$, the two roots of $\det(\xi I-Q_-)$ give $(\sigma-1)+(n-\sigma-1)+2=n$ eigenvalues of $\mathcal{L}(G_S)$ counted with multiplicity. 
	
	For $M_{sym}$, let $V_3 = \{v_{0}\}$. Construct the partition $\Pi_{sym} = \{V_1, V_2, V_3\}$. Extend $x$ and $y$ with $0$ at coordinate $v_{0}$. 
	Since $\sum_{i \in V_1} x_i = 0$ and $\sum_{i \in V_2} y_i = 0$, these extended vectors are  orthogonal to the row vector corresponding to $v_0$ in $M_{sym}$. Thus, they remain eigenvectors of $M_{sym}$, preserving the multiplicities of $n+1$ and $n$. Equivalently, we view $M_{sym}$ with respect to the coordinate partition $V_1\cup V_2\cup\{v_0\}$; this is only a permutation of coordinates and does not change the spectrum.
	
	Let $\widetilde{q}_{ij}$ denote the sum of entries in $M_{sym}$ from row $u \in V_i$ over columns $v \in V_j$.
	For $u \in V_1$, $\widetilde{q}_{13} = -\sqrt{2}$. For $u \in V_2$, $\widetilde{q}_{23} = 0$. For $u \in V_3$, $\widetilde{q}_{31} = |V_1|(-\sqrt{2}) = -\sigma\sqrt{2}$, $\widetilde{q}_{32} = 0$, and $\widetilde{q}_{33} = 2\sigma$. The entries $\widetilde{q}_{11}, \widetilde{q}_{12}, \widetilde{q}_{21}, \widetilde{q}_{22}$ are identical to $Q_-$.	Since $\widetilde{q}_{ij}$ is constant for all $u \in V_i$, $\Pi_{sym}$ is equitable.
	The quotient matrix of $M_{sym}$ over $\Pi_{sym}$ is:
	\begin{equation}\label{eq:quomat2}
	Q_+ = \begin{pmatrix} n-\sigma+1 & -(n-\sigma) & -\sqrt{2} \\ -\sigma & \sigma & 0 \\ -\sigma\sqrt{2} & 0 & 2\sigma \end{pmatrix}. 
	\end{equation}	
	Since $\Pi_{sym}$ is equitable, $\mathrm{Spec}(Q_+)\subseteq \mathrm{Spec}(M_{sym})$.
	Moreover,
	\begin{align*}
	\det(\xi I-Q_+)
	&= \xi^3-(n+2\sigma+1)\xi^2+\sigma (2n+1)\xi \\
	&= \xi\left(\xi^2-(n+2\sigma+1)\xi+\sigma (2n+1)\right).
	\end{align*}
	Hence, the quotient eigenvalues are $0$ and $\eta_\pm$ as in \eqref{eq:eta+-}. Together with the eigenvalue $n+1$ of multiplicity $\sigma-1$ and the eigenvalue $n$ of multiplicity $n-\sigma-1$, this gives $(\sigma-1)+(n-\sigma-1)+3=n+1$ eigenvalues of $M_{sym}$ counted with multiplicity. This completes the proof. 
\end{proof}

\begin{corollary}
	Let $G_S = (K_n)_S$ with $|S| = \sigma$ and $0<\sigma<n$ . Then, at $c=0$ and $q=1$, we have
	\begin{equation}
		\mathcal{M}_1(M_{sym}, 0) - \mathcal{M}_1(\mathcal{L}(G_S), 0) = 2\sigma.
	\end{equation}
\end{corollary}

\begin{proof}
	For $c=0$ and $q=1$, $\mathcal{M}_1(A, 0) = \mathrm{Tr}(A)$ since $\mathcal{L}(G_S)$ and $M_{sym}$ are positive semi-definite.
	Evaluating the trace of the quotient matrices \eqref{eq:quomat1} and \eqref{eq:quomat2} yields 
	\[
	\mathrm{Tr}(Q_-) = n+1 \quad \text{ and } \quad
	\mathrm{Tr}(Q_+) = n+1+2\sigma.
	\]
	Thus, we obtain
	\[
	\mathrm{Tr}(M_{sym}) - \mathrm{Tr}(\mathcal{L}(G_S)) = \mathrm{Tr}(Q_+) - \mathrm{Tr}(Q_-) = 2\sigma. \qedhere
	\]
\end{proof}

\begin{example}\label{Kn_Example_min_max_equality}
	Let $G_S = \widehat{K_n},\: n\geq 2,$ be the complete full-loop graph. Consider its Aç{\i}kme\c{s}e lift $\widetilde{G}$, which is exactly the join $K_1 \nabla 2K_n$. Let $c=0$ and $q=1$.
	Then,
	\[
		\mathrm{Spec}(M_{sym}) = 
		\begin{bmatrix} 
		2n+1 & n+1 & 0 \\
		 1   & n-1 & 1	
		\end{bmatrix} \ \text{ and } \
		\mathrm{Spec}(\mathcal{L}(G_S)) =
		\begin{bmatrix} 
		n+1 & 1 \\
		n-1 & 1	
		\end{bmatrix}.
	\]
	Since $e_{2n+1} = 0$, we have 
	$\mathrm{Spec}_0(M_{sym}) =
	\begin{bmatrix} 
		2n+1 & n+1 \\
		1 	 & n-1
	\end{bmatrix}.$
	Thus, 
	\[
	\min(\mathrm{Spec}_0(M_{sym})) = n+1 =\max(\mathrm{Spec}(\mathcal{L}(G_S))),
	\] 
	the equality \eqref{eq:tightbound2} is satisfied.	
	From the combined spectrum  $\begin{bmatrix} 
		2n+1& n+1  & 1 & 0 \\
		 1  & 2n-2 & 1 & 1 
	\end{bmatrix},$ 
	the upper bound computes to:
	\begin{equation*}
		e_{2n+1} + \sum_{i=1}^n e_i - \sum_{i=n+1}^{2n} e_i = 0 + \Big( (2n+1) + (n-1)(n+1) \Big) - \Big( (n-1)(n+1) + 1 \Big)=2n. 
	\end{equation*}
	On the other hand, the trace difference is
	\begin{equation*}
		\mathcal{M}_1(M_{sym}, 0) - \mathcal{M}_1(\mathcal{L}(G_S), 0) = \Big( 2n+1 + (n-1)(n+1) + 0 \Big) - \Big( (n-1)(n+1) + 1 \Big) =2n,
	\end{equation*}
	verifying the upper bound \eqref{eq:tightbound1} is attained.
\end{example}

	\section{Equivariant Heat Character and Equivariant Laplacian Moments}
	\label{Sect.5}

In this section, we shall adopt the following setting and convention throughout. Let $G_S$ be a self-loop graph of order $N$ and size $m$, with underlying loopless graph $G$. Let $S\subseteq V(G)$ be the set of looped vertices with $\sigma=|S|$ and $s$ be the loop vector. Define the \textit{edge boundary} of $S$ in $G$ by
\begin{equation}\label{eq:edgeboundary}
\partial S := \bigl\{\{u,v\}\in E(G) \mid u\in S,\ v\notin S\bigr\}.
\end{equation}
Moreover, for $v\in V(G)$, define
\begin{equation}\label{eq:bS}
b_S(v):= \#\{u\in V(G) \mid \{u,v\}\in E(G),\ s_u\neq s_v\},
\end{equation}
\begin{equation}\label{eq:kappaS}
\kappa(S):=\sum_{v\in V(G)}b_S(v)^2.
\end{equation}
Indeed, these quantities are motivated by the following observation: \\
\textbf{\textit{Observation}.}
For each $v\in V(G)$,
\[
(\mathcal{L}(G)s)_v
=
\sum_{u:\{u,v\}\in E(G)}(s_v-s_u).
\]
There are two cases:\\ 
Case 1. If $v\in S$, then $s_v=1$, and the nonzero summands occur exactly when $u\notin S$, resulting in each such summand being $1$. Hence, 
$(\mathcal{L}(G)s)_v=b_S(v)$ for $v\in S.$\\	
Case 2. If $v\notin S$, then $s_v=0$, and the nonzero summands occur exactly when $u\in S,$ resulting in each such summand being $-1$. Hence,
$(\mathcal{L}(G)s)_v=-b_S(v)$ for $v\notin S.$

This implies that
\begin{equation}\label{eq:kappaS2}
\|\mathcal{L}(G)s\|^2 = \sum_{v\in V(G)}b_S(v)^2 = \kappa(S).
\end{equation}

\begin{definition}
	\label{equivheatchar}
	Let $\widetilde{G}$ be the A\c{c}{\i}kme\c{s}e lift of $G_S$ and $T$ be the involution matrix induced by the $\mathbb{Z}_2$-action on $\widetilde{G}$.
	Let $\rho:\mathbb{Z}_2\to \mathrm{GL}(\mathbb{R}^{\widetilde{V}})$ be the representation induced by the lift involution, with $\rho(1)=I$ and $\rho(g)=T$. For $t>0$, define the \textit{equivariant heat character} of $h$ as $\chi_t^{\widetilde{G}}:\mathbb{Z}_2 \to \mathbb{R}$ by
	\[
	\chi_t^{\widetilde{G}}(h) := \mathrm{Tr}\!\left(e^{-t\mathcal{L}(\widetilde{G})}\rho(h)\right),
	\qquad h\in\mathbb{Z}_2.
	\]
	In particular,
	\begin{equation}\label{eq:equivheatchar1}
	\chi_t^{\widetilde{G}}(1) =	\mathrm{Tr}\!\left(e^{-t\mathcal{L}(\widetilde{G})}\right),
	\qquad
	\chi_t^{\widetilde{G}}(g) =	\mathrm{Tr}\!\left(e^{-t\mathcal{L}(\widetilde{G})}T\right).
	\end{equation}
\end{definition}

\begin{theorem}
	\label{thm:isotypic_heat_decomposition}
	For every $t>0$, the equivariant heat characters split into the following traces:
	\begin{equation}\label{eq:equivheatchar2}
	\chi_t^{\widetilde{G}}(1) = \mathrm{Tr}(e^{-tM_{sym}}) + \mathrm{Tr}(e^{-t\mathcal{L}(G_S)}),
	\end{equation}
	\begin{equation}\label{eq:equivheatchar3}
	\chi_t^{\widetilde{G}}(g) = \mathrm{Tr}(e^{-tM_{sym}}) - \mathrm{Tr}(e^{-t\mathcal{L}(G_S)}).
	\end{equation}
	Consequently,
	\begin{equation}\label{eq:equivheatchar4}
	\mathrm{Tr}(e^{-tM_{sym}}) = \frac{\chi_t^{\widetilde{G}}(1)+\chi_t^{\widetilde{G}}(g)}{2},
	\qquad
	\mathrm{Tr}(e^{-t\mathcal{L}(G_S)}) = \frac{\chi_t^{\widetilde{G}}(1)-\chi_t^{\widetilde{G}}(g)}{2}.
	\end{equation}
	Moreover, for every $t>0$,
	\begin{equation}
	-\chi_t^{\widetilde{G}}(1) \le	\chi_t^{\widetilde{G}}(g) \le \chi_t^{\widetilde{G}}(1).
	\end{equation}
	Equivalently,
	$\chi_t^{\widetilde{G}}(1)\pm\chi_t^{\widetilde{G}}(g)\ge 0$.	
\end{theorem}

\begin{proof}
	By Theorem \ref{orthogonalequiv}(ii), there exists an orthogonal matrix $U$ satisfying 
	\[
	U^\top\mathcal{L}(\widetilde{G})U= \mathrm{diag}(\mathcal{L}(G_S),M_{sym}).
	\]
	Under the same orthogonal change of basis, by Remark~\ref{REMARKorthogonalequiv}, $U^\top T U=\mathrm{diag}(-I_N, I_{N+1}).$ Since matrix exponentials respect similarity, we obtain
	\[
	U^\top e^{-t\mathcal{L}(\widetilde{G})}U = e^{-tU^\top\mathcal{L}(\widetilde{G})U} =
	\begin{pmatrix}
		e^{-t\mathcal{L}(G_S)}&0\\
		0&e^{-tM_{sym}}
	\end{pmatrix}.
	\]
	By the invariance of trace under similarity, we obtain  \eqref{eq:equivheatchar2}.
	Using the same orthogonal basis,
	\begin{align*}
	U^\top e^{-t\mathcal{L}(\widetilde{G})}TU
	&= 	(U^\top e^{-t\mathcal{L}(\widetilde{G})} U)(U^\top T U) \\
	&= \begin{pmatrix}
		e^{-t\mathcal{L}(G_S)}&0\\
		0&e^{-tM_{sym}}
	\end{pmatrix}
	\begin{pmatrix}
		-I_N&0\\
		0&I_{N+1}
	\end{pmatrix} \\
	&=\begin{pmatrix}
		-e^{-t\mathcal{L}(G_S)}&0\\
		0&e^{-tM_{sym}}
	\end{pmatrix}.
	\end{align*}
	Since $\mathrm{Tr}\left(U^\top e^{-t\mathcal{L}(\widetilde{G})}TU \right) = \mathrm{Tr}\left( e^{-t\mathcal{L}(\widetilde{G})}T \right),$ we get  \eqref{eq:equivheatchar3}. The identities \eqref{eq:equivheatchar4} follow immediately.	Finally, since $M_{sym}$ and $\mathcal{L}(G_S)$ are real symmetric positive semidefinite	matrices, $e^{-tM_{sym}}$ and $e^{-t\mathcal{L}(G_S)}$ are also positive semidefinite, and their traces are nonnegative. Therefore, for every $t>0$,
	\[
	\chi_t^{\widetilde{G}}(1)+\chi_t^{\widetilde{G}}(g)	= 2\mathrm{Tr}(e^{-tM_{sym}}) \ge 0,
	\]
	\[
	\chi_t^{\widetilde{G}}(1)-\chi_t^{\widetilde{G}}(g)	= 2\mathrm{Tr}(e^{-t\mathcal{L}(G_S)}) \ge 0,
	\]
	which is equivalent to
	\[
	-\chi_t^{\widetilde{G}}(1) \le \chi_t^{\widetilde{G}}(g) \le \chi_t^{\widetilde{G}}(1). \qedhere
	\]
\end{proof}

In the following, we provide a stability estimate for the equivariant heat character, which essentially implies that
for fixed $t>0$ and $h\in\mathbb{Z}_2$, the map
\[
\mathcal{L}(\widetilde{G}) \longmapsto \chi_t^{\widetilde{G}}(h) = \mathrm{Tr}\!\left(e^{-t\mathcal{L}(\widetilde{G})}\rho(h)\right)
\]
is Lipschitz continuous with respect to the trace norm on the lifted Laplacian.

\begin{proposition}
	\label{prop:heat_character_lipschitz}
	Let $G_{S,1}$ and $G_{S,2}$ be two self-loop graphs of the same order $N$. Let $\widetilde{G}_1$ and $\widetilde{G}_2$ be their A\c{c}{\i}kme\c{s}e lifts, respectively, with the same vertex ordering $1,\dots,N, N+1,\dots,2N+1.$	Let $T$ be the common involution matrix which swaps the two copies and fixes the middle vertex. Then, for every $h\in\mathbb{Z}_2$ and every $t>0$,
	\[
	\left|\chi_t^{\widetilde{G}_1}(h)-\chi_t^{\widetilde{G}_2}(h)\right| \le t\|\mathcal{L}(\widetilde{G}_1)-\mathcal{L}(\widetilde{G}_2)\|_1,
	\]
	where $\|\cdot\|_1$ denotes the trace norm.
\end{proposition}

\begin{proof}
	In this proof, for notational simplicity, we write $L_i=\mathcal{L}(\widetilde{G}_i), i=1,2.$ Since the two lifts are under the same vertex ordering, the involution 	matrix $T$ is the same for both lifts. Hence, the representation $\rho:\mathbb{Z}_2\to \mathrm{GL}(\mathbb{R}^{2N+1})$ is common to both, with $\rho(1)=I$ and $\rho(g)=T.$	Therefore, by Definition \ref{equivheatchar},
	\[
	\chi_t^{\widetilde{G}_i}(h)	= \mathrm{Tr}\!\left(e^{-tL_i}\rho(h)\right), \qquad i=1,2.
	\]	
	By 	Duhamel's principle \cite[Chapter~III, Sect.~1]{engel2000oneparameter}, we have
	\[
	e^{-tL_1}-e^{-tL_2}	= -\int_0^t e^{-(t-r)L_1}(L_1-L_2)e^{-rL_2}\,dr.
	\]
	Since $L_1$ and $L_2$ are real symmetric positive semidefinite matrices, for every
	$a\ge 0$,
	\[
	\|e^{-aL_i}\| = \max_{\lambda \in \mathrm{Spec}(L_i)}e^{-a\lambda} \le 1, \qquad i=1,2.
	\]
	Using the trace-norm inequality, we obtain
	\begin{align*}
		\|e^{-tL_1}-e^{-tL_2}\|_1
		&\le \int_0^t \|e^{-(t-r)L_1}(L_1-L_2)e^{-rL_2}\|_1\,dr\\
		&\le \int_0^t \|e^{-(t-r)L_1}\|\,\|L_1-L_2\|_1\,\|e^{-rL_2}\|\,dr\\
		&\le \int_0^t \|L_1-L_2\|_1\,dr\\
		&= t\|L_1-L_2\|_1.
	\end{align*}
	
	Now, observe that
	\[
	\chi_t^{\widetilde{G}_1}(h)-\chi_t^{\widetilde{G}_2}(h)	= \mathrm{Tr}\!\left((e^{-tL_1}-e^{-tL_2})\rho(h)\right).
	\]
	Since $\rho(h)$ is orthogonal, $\|\rho(h)\|=1$. By the trace inequality, we obtain
	\[
	\left|\chi_t^{\widetilde{G}_1}(h)-\chi_t^{\widetilde{G}_2}(h)\right| 
	\le	\|e^{-tL_1}-e^{-tL_2}\|_1\|\rho(h)\|
	\le	t\|L_1-L_2\|_1.  \qedhere
	\]
\end{proof}

\vspace{0.5em}

Since the matrix exponential is given by its power series and the matrices are finite-dimensional, linearity of trace gives
\[
\chi_t^{\widetilde{G}}(1)
= \mathrm{Tr}\left(e^{-t\mathcal{L}(\widetilde{G})}\right)
= \sum_{k=0}^{\infty}
\frac{(-t)^k}{k!}
\mathrm{Tr}\!\left(\mathcal{L}(\widetilde{G})^k\right).
\]
In other words, the ordinary Laplacian spectral moments $M_k(\mathcal{L}(\widetilde{G}))$ in Theorem \ref{thm:spectral_moments} is the $k$-th heat coefficient of $\chi_t^{\widetilde{G}}(1).$ What about $\chi_t^{\widetilde{G}}(g)?$ This motivates the following definition where we incorporate the group action $T.$

\begin{definition}	\label{equivLaplacianmoment}
	Let $\widetilde{G}$ be the A\c{c}{\i}kme\c{s}e lift of $G_S.$ For $k\ge 0$, define the $k$-th \textit{equivariant Laplacian moment} of the
	A\c{c}{\i}kme\c{s}e lift by
	\[
	\Theta_k(\widetilde{G},T) := \mathrm{Tr}\!\left(\mathcal{L}(\widetilde{G})^kT\right).
	\]
\end{definition}

More precisely, 
\[
\chi_t^{\widetilde{G}}(g)
= \mathrm{Tr}\!\left(\sum_{k=0}^{\infty}\frac{(-t)^k}{k!}\mathcal{L}(\widetilde{G})^kT\right)
= \sum_{k=0}^{\infty}\frac{(-t)^k}{k!}\Theta_k(\widetilde{G},T).
\]
Thus, $\Theta_k(\widetilde{G},T)$ is the term associated to the $k$-th heat coefficient of the nontrivial equivariant heat character. Moreover, by Remark~\ref{REMARKorthogonalequiv}, the involution acts by $-I_N$ on the anti-symmetric block and by $I_{N+1}$ on the symmetric block implies that
\begin{equation}\label{eq:theta-block-difference}
	\Theta_k(\widetilde{G},T)
	= \mathrm{Tr}(M_{sym}^k) - \mathrm{Tr}(\mathcal{L}(G_S)^k).
\end{equation}
In contrast, the ordinary Laplacian spectral moment is the sum \eqref{eq:laplacianmoment1}
\[
M_k(\mathcal{L}(\widetilde{G})) = \mathrm{Tr}(M_{sym}^k)+\mathrm{Tr}(\mathcal{L}(G_S)^k).
\]
Consequently, take its difference with \eqref{eq:theta-block-difference}, we have
\[
M_k(\mathcal{L}(\widetilde{G})) = 2M_k(\mathcal{L}(G_S)) + \Theta_k(\widetilde{G},T).
\]
This indicates that the equivariant Laplacian moments $\Theta_k(\widetilde{G},T)$ are some kind of \textit{correction terms} of the ordinary Laplacian moments associated to the lift.

\begin{theorem} \label{thm:heat_coefficients_loop_placement}
	The first five terms of $\Theta_k(\widetilde{G},T)$ are
	\[
	\Theta_0(\widetilde{G},T)= 1,\quad 
	\Theta_1(\widetilde{G},T)= 2\sigma, \quad 	
	\Theta_2(\widetilde{G},T)= 4\sigma^2 + 4\sigma,
	\]
	\[
	\Theta_3(\widetilde{G},T) = 8\sigma^3 + 12\sigma^2 + 6\sigma + 6|\partial S|,
	\]
	\[
	\Theta_4(\widetilde{G},T) = 16\sigma^4 + 32\sigma^3+24\sigma^2+8\sigma + 16(\sigma+1)|\partial S| + 8\kappa(S).
	\]
	Consequently, $\chi_t^{\widetilde{G}}(g)$ determines $\sigma$, $|\partial S|$, and
	$\kappa(S)$ from its Taylor expansion at $t=0$.
\end{theorem}

\begin{proof}
	The identity $\Theta_0=\mathrm{Tr}(T)=1$ follows from the dimensions of the symmetric and anti-symmetric subspaces.	For $k\ge 1$, we use \eqref{eq:theta-block-difference}. The first three formulae follow immediately from \eqref{eq:theta-block-difference} with \eqref{eq:m1},\eqref{eq:m2}, and \eqref{eq:m3}, respectively.
	Since $\mathcal{L}(G_S)=\mathcal{L}(G)+\mathrm{diag}(s)$, we have
	\[
	s^\top\mathcal{L}(G_S)s=|\partial S|+\sigma,
	\]
	giving the formula for $\Theta_3$. For the fourth moment, direct block multiplication gives
	\[
	\Theta_4 = 16\sigma^4+32\sigma^3+8\sigma^2 + 16\sigma\,s^\top\mathcal{L}(G_S)s + 8s^\top\mathcal{L}(G_S)^2s.
	\]
	Then, $\Theta_4$ follows from the observation that
	\begin{align*}
	s^\top\mathcal{L}(G_S)^2s
	&= \|\mathcal{L}(G)s+s\|^2 \\
	&= 	\|\mathcal{L}(G)s\|^2 + 2s^\top\mathcal{L}(G)s + s^\top s \\
	&= \kappa(S)+2|\partial S|+\sigma,
	\end{align*}
	where the first summand follows from \eqref{eq:kappaS2}.
 \end{proof}

Since $\mathcal{L}(G_S)$ is a real symmetric $N\times N$ matrix, it defines a bounded self-adjoint operator on $\mathbb{R}^N.$ As inspired by the Weyl $m$-function (cf. \cite[\S 2]{teschl2000jacobi}) and the matrix element of the resolvent of a bounded operator (cf. \cite[pp 115, \S 5]{demuth2005determining}), we introduce the following definition.

\begin{definition}
	\label{loop_resolvent_form}
	Let $G_S$ be a self-loop graph. For $p>0$, define the \emph{loop resolvent}\footnote{Strictly speaking, this is the quadratic form or the matrix element of the resolvent operator evaluated at the loop vector $s$, we refer it in this short-form for simplicity.} of $G_S$ as a function $R_{G,S}(p) : (0,\infty) \to \mathbb{R}$ by
	\begin{equation}\label{eq:loop_resolvent_form}
	R_{G,S}(p):=s^\top(pI+\mathcal{L}(G_S))^{-1}s
	\end{equation}
	where $s$ is the loop vector and $(pI+\mathcal{L}(G_S))^{-1}$ is the positive resolvent of $\mathcal{L}(G_S)$ at $-p$.
\end{definition}

To illustrate this definition, let's consider a simple non-trivial example: 

\begin{example}	Let $G_S$ be the complete self-loop graph $(K_N)_S$, where $1\le \sigma=|S|\le N-1$. It is clear that
	\[
	\mathcal{L}(G_S)=NI_N-J_N+\mathrm{diag}(s).
	\]
	Let $x=(pI+\mathcal{L}(G_S))^{-1}s. $ By symmetry of $K_N$ with respect to $S$, the vector $x$ is constant on $S$ and constant on $V(K_N)\setminus S$. So, we may write
	\[
	x_i = 
	\begin{cases}
	a, &i\in S, \\	
	b, &i\notin S.
	\end{cases}
	\]
	Now, set $r=N-\sigma$. Since $(pI+\mathcal{L}(G_S))x=s,$
	we get the two equations
	\[
	\begin{cases}
	(p+N+1)a-(\sigma a+rb)=1, \\
	(p+N)b-(\sigma a+rb)=0.
	\end{cases} 
	= 
	\begin{cases}
	(p+r+1)a-rb=1, \\
	-\sigma a+(p+\sigma)b=0.
	\end{cases} 
	\]
	Solving these simultaneous equations, we obtain
	\[
	a = \frac{p+\sigma}{(p+\sigma)(p+r+1)-r\sigma} = \frac{p+\sigma}{p^2+(N+1)p+\sigma}.
	\]
	It follows that
	\[
	R_{K_N,S}(p)
	= s^\top(pI+\mathcal{L}(G_S))^{-1}s
	= s^\top x
	= \sum_{i\in S}x_i
	= \sigma a
	= \frac{\sigma (p+\sigma)}{p^2+(N+1)p+\sigma}.
	\]
\end{example}

In the following, we derive a Laplace-transform identity and a determinantal formula that relate the equivariant heat character  $\chi^{\widetilde{G}}_t(g)$ and the loop resolvent $R_{G,S}(p).$

\begin{theorem}	\label{thm:equivariant_resolvent}
	For $p>0$,
	\begin{equation}\label{eq:equivariant_resolvent1}
	\int_0^\infty e^{-pt}\chi_t^{\widetilde{G}}(g)\,dt = \frac{d}{dp}\log\left(p + 2\sigma - 2R_{G,S}(p)\right).
	\end{equation}
	Moreover, in the determinant form, we have
	\begin{equation}\label{eq:equivariant_resolvent2}
	\frac{\det(pI+M_{sym})}{\det(pI+\mathcal{L}(G_S))} = p + 2\sigma - 2s^\top(pI+\mathcal{L}(G_S))^{-1}s.
	\end{equation}
\end{theorem}

\begin{proof}
	By Theorem~\ref{thm:isotypic_heat_decomposition}, we have $\chi_t^{\widetilde{G}}(g) = \mathrm{Tr}(e^{-tM_{sym}})-\mathrm{Tr}(e^{-t\mathcal{L}(G_S)}).$	
	Multiplying by $e^{-pt}$ gives
	\[
	e^{-pt}\chi_t^{\widetilde{G}}(g) = \mathrm{Tr}(e^{-t(pI+M_{sym})}) - \mathrm{Tr}(e^{-t(pI+\mathcal{L}(G_S))}).
	\]
	Since $p>0$ and both $M_{sym}$ and $\mathcal{L}(G_S)$ are positive semidefinite, the matrices $pI+M_{sym}$ and $pI+\mathcal{L}(G_S)$ are positive definite. In particular, if $A$ is either $M_{sym}$ or $\mathcal{L}(G_S)$, then
	\[
	\left|\mathrm{Tr}\left(e^{-t(pI+A)}\right)\right| \le \dim(A)e^{-pt}.
	\]
	Thus, the scalar functions $t \mapsto \mathrm{Tr}\left(e^{-t(pI+A)}\right)$ are absolutely integrable on $[0,\infty)$, namely,
	\[
	\int_0^\infty \left| \mathrm{Tr}\left(e^{-t(pI+A)}\right)\right| dt \leq \dim(A) \int_0^\infty e^{-pt}dt = \frac{\dim(A)}{p} < \infty.
	\]	
	Moreover, since all matrices are finite-dimensional, the matrix integral is	entrywise, and the trace is a finite sum of diagonal entries. Hence,
	\[
	\int_0^\infty \mathrm{Tr}(e^{-t(pI+A)})\,dt	= \mathrm{Tr}\!\left(\int_0^\infty e^{-t(pI+A)}\,dt\right).
	\]
	By a semigroup resolvent formula \cite[Chapter~II, Sect.~1]{engel2000oneparameter}, applied to the semigroup $e^{-tA}$, we have
	\[
	\int_0^\infty e^{-t(pI+A)}\,dt = (pI+A)^{-1}.
	\]	
	Consequently, 
	\begin{equation}\label{eq:inteptchi}
	\int_0^\infty e^{-pt}\chi_t^{\widetilde{G}}(g)\,dt = \mathrm{Tr}((pI+M_{sym})^{-1})	- \mathrm{Tr}((pI+\mathcal{L}(G_S))^{-1}).
	\end{equation}
	Recall that for an invertible differentiable matrix family $A(p)$, Jacobi's formula gives
	\[
	\frac{d}{dp}\log\det A(p) = \mathrm{Tr}\!\left(A(p)^{-1}A'(p)\right).
	\] 
	With $A(p)=pI+M_{sym},$ we have
	\begin{equation}\label{eq:logdet1}
	\frac{d}{dp}\log\det(pI+M_{sym}) = \mathrm{Tr}((pI+M_{sym})^{-1}).
	\end{equation}
	Similarly, for $\mathcal{L}(G_S),$ we deduce that
	\begin{equation}\label{eq:logdet2}
	\frac{d}{dp}\log\det(pI+\mathcal{L}(G_S)) = \mathrm{Tr}((pI+\mathcal{L}(G_S))^{-1}).
	\end{equation}
	Taking the difference \eqref{eq:logdet1} and \eqref{eq:logdet2}, together with linearity of differentiation, yields
	\begin{equation}\label{eq:logdet3}
	\mathrm{Tr}((pI+M_{sym})^{-1}) - \mathrm{Tr}((pI+\mathcal{L}(G_S))^{-1}) = \frac{d}{dp} \log\frac{\det(pI+M_{sym})}{\det(pI+\mathcal{L}(G_S))}.
	\end{equation}
	Combining \eqref{eq:inteptchi} and \eqref{eq:logdet3} gives
	\begin{equation}\label{eq:equivariant_resolvent3}
	\int_0^\infty e^{-pt}\chi_t^{\widetilde{G}}(g)\,dt = \frac{d}{dp}\log\frac{\det(pI+M_{sym})}{\det(pI+\mathcal{L}(G_S))}.
	\end{equation}
		
	It remains to compute the determinant ratio. By Theorem \ref{orthogonalequiv}(ii), we get
	\begin{equation}\label{eq:pI+Msym}
	pI+M_{sym} =
	\begin{pmatrix}
		p+2\sigma&-\sqrt{2}s^\top\\
		-\sqrt{2}s&pI+\mathcal{L}(G_S)
	\end{pmatrix}.
	\end{equation}
	Since $p>0$ and $\mathcal{L}(G_S)$ is positive semidefinite, $pI+\mathcal{L}(G_S)$ is invertible. Applying the Schur complement determinant formula to \eqref{eq:pI+Msym} yields 
	\begin{align*}
		\det(pI+M_{sym})
		&= \det(pI+\mathcal{L}(G_S))\left(p+2\sigma	- (-\sqrt{2}s^\top)(pI+\mathcal{L}(G_S))^{-1}(-\sqrt{2}s)\right)\\
		&= \det(pI+\mathcal{L}(G_S))\left(p+2\sigma	- 2s^\top(pI+\mathcal{L}(G_S))^{-1}s\right).
	\end{align*}
	Dividing both sides by $\det(pI+\mathcal{L}(G_S))\neq 0$, we obtain \eqref{eq:equivariant_resolvent2}, for which \eqref{eq:equivariant_resolvent1} follows immediately by combining  \eqref{eq:equivariant_resolvent2}, \eqref{eq:equivariant_resolvent3}, and \eqref{eq:loop_resolvent_form}.
\end{proof}

\begin{theorem}
	\label{thm:loop_resolvent_spectral_measure_and_moments}
	Let $G_S$ be a self-loop graph and $s$ be the loop vector. Let $R_{G,S}(p)$ be the loop resolvent of $G_S.$ Furthermore, consider $\mathcal{L}(G_S)=\sum_{\lambda}\lambda P_\lambda$, where $P_\lambda$ denotes the orthogonal projection onto the $\lambda$-eigenspace. Then, the following hold:
	
	\begin{enumerate}[(i)]
		\item $R_{G,S}(p)$ has the spectral representation
		\begin{equation} \label{eq:RGSspecrep}
		R_{G,S}(p) =\sum_{\lambda} \frac{\|P_\lambda s\|^2}{p+\lambda}.
		\end{equation}
		Consequently, the right side gives a rational continuation of $R_{G,S}$ to $\mathbb{C}$, with poles precisely at those points $p=-\lambda$ for which $P_\lambda s\neq 0$.
		
		\item $R_{G,S}(p)$ has the Laurent expansion at $p=\infty$
		\begin{equation} \label{eq:RGSLaurent}
		R_{G,S}(p) = \sum_{k=0}^{\infty} \frac{(-1)^k m_k(G,S)}{p^{k+1}},
		\end{equation}
		valid for all sufficiently large $|p|$, where $m_k(G,S):=s^\top\mathcal{L}(G_S)^k s$ is the $k$-th loop moment.	
		\end{enumerate}
\end{theorem}

\begin{proof}
	We first prove the spectral representation. Since $\mathcal{L}(G_S)=\sum_{\lambda}\lambda P_\lambda,$ functional calculus gives
	\[
	(pI+\mathcal{L}(G_S))^{-1} = \sum_{\lambda}\frac{1}{p+\lambda}P_\lambda.
	\]
	Since each $P_\lambda$ is an orthogonal projection, we have $P_\lambda=P_\lambda^\top=P_\lambda^2.$
	Hence,
	$s^\top P_\lambda s
	= s^\top P_\lambda^2s
	= (P_\lambda s)^\top(P_\lambda s)
	= \|P_\lambda s\|^2.
	$
	Therefore,
	\[
	R_{G,S}(p)
	= s^\top(pI+\mathcal{L}(G_S))^{-1}s
	= \sum_{\lambda}
	\frac{s^\top P_\lambda s}{p+\lambda} = \sum_{\lambda}
	\frac{\|P_\lambda s\|^2}{p+\lambda}.
	\]
	The right-hand side is a finite sum of rational functions and agrees with $R_{G,S}(p)$ for every $p>0$. Hence, it defines a rational continuation of	$R_{G,S}$ to $\mathbb C$. Since $\|P_\lambda s\|^2>0$ if and only if	$P_\lambda s\neq0$, the poles of this continuation are precisely those
	points $p=-\lambda$ for which $P_\lambda s\neq0$.
	
	As for the second claim, consider the rational continuation obtained in	Part~(i). Since $\mathcal{L}(G_S)$ is a finite symmetric matrix, the Neumann expansion is valid whenever $|p|$ is larger than	the spectral radius of $\mathcal{L}(G_S)$. Hence,
	\[
	(pI+\mathcal{L}(G_S))^{-1}
	= \frac{1}{p}\left(I+\frac{1}{p}\mathcal{L}(G_S)\right)^{-1}
	= \frac{1}{p}\sum_{k=0}^{\infty}(-1)^k\frac{\mathcal{L}(G_S)^k}{p^k}.
	\]
	Multiplying on the left and right by $s^\top$ and $s$, respectively, gives \eqref{eq:RGSLaurent}. Since a rational function has a unique Laurent expansion at infinity, the coefficients of this expansion determine all moments $m_k(G,S)$.
\end{proof}

\begin{remark}
	Let $\mu$ be a finite positive measure on $\mathbb{R}$. Its Stieltjes transform (cf. e.g., \cite[\S 2, Sect.~2.1]{teschl2000jacobi}) is the function \[ \mathcal{S}_{\mu}(z) = \int_{\mathbb{R}}\frac{1}{\lambda-z}\,d\mu(\lambda), \quad z\notin \operatorname{supp}(\mu). \] Equivalently, if the support of $\mu$ lies in $[0,\infty)$, then for $p>0,$ one often writes the same transform on the negative real axis as 
	\[ 
	F_{\mu}(p) = \int_{[0,\infty)}\frac{1}{p+\lambda}\,d\mu(\lambda). 
	\] 
	This is precisely the form appearing in the loop resolvent. Indeed, by the spectral theorem, the self-loop Laplacian has the spectral decomposition $ \mathcal{L}(G_S)=\sum_{\lambda}\lambda P_\lambda, $ where $P_\lambda$ is the orthogonal projection onto the $\lambda$-eigenspace. The loop vector $s$ therefore determines the finite positive measure 
	\[ 
	\mu_{G,S} = \sum_{\lambda} \|P_\lambda s\|^2\delta_\lambda. 
	\]  
	Hence, it follows from Theorem \ref{thm:loop_resolvent_spectral_measure_and_moments} that
	\[ 
	R_{G,S}(p) = \sum_{\lambda} \frac{\|P_\lambda s\|^2}{p+\lambda} = \int_{[0,\infty)} \frac{1}{p+\lambda}\,d\mu_{G,S}(\lambda). 
	\] 
	This means that $R_{G,S}$ is the Stieltjes transform of the spectral measure $\mu_{G,S}$. This interpretation is useful because the Stieltjes transform encodes the locations and weights of the atoms of $\mu_{G,S}$. In the present finite graph setting, the poles of $R_{G,S}$ occur exactly at $p=-\lambda$, where $\lambda$ is an eigenvalue of $\mathcal{L}(G_S)$ satisfying $P_\lambda s\neq 0$, and the corresponding residue is $\|P_\lambda s\|^2$.
\end{remark}

	From \eqref{eq:RGSspecrep}, $R_{G,S}$ is a rational function, thus it is continuous and differentiable on $(0,\infty).$	Now, suppose that $G_S$ is pseudo-connected. Since $\mathcal{L}(G_S)$ is invertible,  we have $\lim_{p \to 0} (pI+\mathcal{L}(G_S))^{-1} = \mathcal{L}(G_S)^{-1}.$
	Hence, $R_{G,S}$ admits a continuous extension to $p=0$, given by
	\[
	R_{G,S}(0) = \lim_{p \to 0}R_{G,S}(p) = s^\top\mathcal{L}(G_S)^{-1}s.
	\]	
    On the other hand, since $\det(M_{sym})=0$ and $\det(\mathcal{L}(G_S))\neq0,$ when $p \to 0$ and applying  \eqref{eq:equivariant_resolvent2}, a simple deduction yields the following result.
	\begin{corollary} \label{cor:Linv}
	For a pseudo-connected self-loop graph $G_S$ and loop vector $s,$ it holds $s^\top\mathcal{L}(G_S)^{-1}s=\sigma.$
	\end{corollary}

\begin{theorem}
	\label{thm:visible_spectral_subsystem}
	Let $\mathcal{K}_{G,S} := \mathrm{span}\{s,\mathcal{L}(G_S)s,\mathcal{L}(G_S)^2s,\ldots\}$ be the Krylov subspace generated by the loop vector $s$.
	Then, 
	\begin{enumerate}[(i)] 
		\item The subspace $\mathcal{K}_{G,S}$ is the smallest $\mathcal{L}(G_S)$-invariant subspace containing $s$. 
		\item For every $p>0$, 
		$(pI+\mathcal{L}(G_S))^{-1}s\in \mathcal{K}_{G,S}.$ 
		\item The loop resolvent $R_{G,S}(p)$ determines the orthogonal equivalence class of the triple 
		$\left(\mathcal{K}_{G,S}, \mathcal{L}(G_S)|_{\mathcal{K}_{G,S}}, s \right).$
		\item Two self-loop graphs $(G,S)$ and $(H,T)$ have the same loop resolvents, 
		\[ 
		R_{G,S}(p)=R_{H,T}(p), \quad \forall p>0, 
		\] 
		if and only if the triples $ \left( \mathcal{K}_{G,S}, \mathcal{L}(G_S)|_{\mathcal{K}_{G,S}}, s_S \right) $ and $ \left( \mathcal{K}_{H,T}, \mathcal{L}(H_T)|_{\mathcal{K}_{H,T}}, s_T \right)$ are orthogonally equivalent.
	\end{enumerate}
\end{theorem}

\begin{proof} By definition, it is clear that $\mathcal{L}(G_S)\mathcal{K}_{G,S} \subseteq \mathcal{K}_{G,S},$ i.e., $\mathcal{K}_{G,S}$ is $\mathcal{L}(G_S)$-invariant. If $\mathcal{W}$ is any $\mathcal{L}(G_S)$-invariant subspace containing $s$, then it must also contain $ \mathcal{L}(G_S)s,$ $\mathcal{L}(G_S)^2s,$ $\ldots.$ Hence, $ \mathcal{K}_{G,S}\subseteq \mathcal{W}, $ which shows that $\mathcal{K}_{G,S}$ is the smallest $\mathcal{L}(G_S)$-invariant subspace containing $s$. Next, since $\mathcal{K}_{G,S}$ is invariant under $\mathcal{L}(G_S)$, it is also invariant under $pI + \mathcal{L}(G_S),$ i.e., $(pI + \mathcal{L}(G_S))\mathcal{K}_{G,S} \subseteq \mathcal{K}_{G,S}.$ Since $pI + \mathcal{L}(G_S)$ is invertible, so is $pI + \mathcal{L}(G_S)|_{\mathcal{K}_{G,S}}.$ Thus, 
\[
(pI + \mathcal{L}(G_S))^{-1} \mathcal{K}_{G,S} \subseteq \mathcal{K}_{G,S}.
\]
Since $s \in \mathcal{K}_{G,S}, $ we have $(pI+\mathcal{L}(G_S))^{-1}s\in \mathcal{K}_{G,S}.$

Since $\mathcal{L}(G_S)$ is symmetric, the orthogonal complement $\mathcal{K}_{G,S}^\perp$ is also $\mathcal{L}(G_S)-$invariant, i.e., $\mathcal{L}(G_S)\mathcal{K}_{G,S}^\perp \subseteq \mathcal{K}_{G,S}^\perp.$ This shows that  
\[
\mathbb{R}^{V(G)}= \mathcal{K}_{G,S} \oplus \mathcal{K}_{G,S}^\perp.
\]
Thus, all $\mathcal{L}(G_S),$ $pI+\mathcal{L}(G_S),$ and $(pI+\mathcal{L}(G_S))^{-1} $ split accordingly. Any eigenvector lying entirely in $\mathcal{K}_{G,S}^{\perp}$ is orthogonal to $s$ and therefore has zero weight in $R_{G,S}$.

Since the coefficients $\|P_\lambda s\|^2$ are nonnegative, the  expansion of $R_{G,S}(p)$ in Theorem \ref{thm:loop_resolvent_spectral_measure_and_moments}(i) determines exactly the eigenvalues $\lambda$ such that $P_\lambda s\neq 0,$ together with their weights $\alpha_\lambda=\|P_\lambda s\|^2.$ For every $k\geq 0$,
\[
\mathcal{L}(G_S)^k s = \sum_{\lambda}\lambda^k P_\lambda s,
\]
so $\mathcal{K}_{G,S}$ is contained in the span of the nonzero vectors $P_\lambda s$, i.e., 
\[
\mathcal{K}_{G,S} = \bigoplus_{\lambda:\,P_\lambda s\neq 0} \mathrm{span}\{P_\lambda s\}.
\] 
Conversely, since the set of eigenvalues is finite, each projection $P_\lambda s$ can be obtained from $s,\ \mathcal{L}(G_S)s,\ \mathcal{L}(G_S)^2s,\ldots$ by polynomial interpolation on the eigenvalues. Hence, the two spans coincide.

Observe that the restriction $\mathcal{L}(G_S)|_{\mathcal{K}_{G,S}}$ acts by multiplication by $\lambda$ over $\mathrm{span}\{P_\lambda s\}$, while the loop vector decomposes as $s=\sum_{\lambda:\,P_\lambda s\neq 0}P_\lambda s$ with weight $\alpha_\lambda.$ Thus, the triple $\left(\mathcal{K}_{G,S},\mathcal{L}(G_S)|_{\mathcal{K}_{G,S}},s\right)$ is determined, up to orthogonal equivalence, by the finite weighted spectral data
$\{(\lambda,\alpha_\lambda) \mid \alpha_\lambda>0\}.$ The claim (iii) follows.

Now, to show (iv), suppose the triples of $(G,S)$ and $(H,T)$ are orthogonally equivalent. Thus, there exists an orthogonal map
$U:\mathcal{K}_{G,S}\to \mathcal{K}_{H,T}$ such that
\[
Us_S = s_T
\quad
\text{ and }
\quad 
U\left(\mathcal{L}(G_S)|_{\mathcal{K}_{G,S}}\right) = \left(\mathcal{L}(H_T)|_{\mathcal{K}_{H,T}}\right)U,
\]
\[
U\left(pI+\mathcal{L}(G_S)|_{\mathcal{K}_{G,S}}\right)^{-1} = \left(pI+\mathcal{L}(H_T)|_{\mathcal{K}_{H,T}}\right)^{-1}U.
\]
 Using these, we deduce that
\begin{align*}
R_{G,S}(p)
&=\langle s_S, \left(pI+\mathcal{L}(G_S)|_{\mathcal{K}_{G,S}}\right)^{-1}s_S \rangle \\
&= \langle U s_S, U\left(pI+\mathcal{L}(G_S)|_{\mathcal{K}_{G,S}}\right)^{-1}s_S \rangle \\
&= \langle s_T, \left(pI+\mathcal{L}(H_T)|_{\mathcal{K}_{H,T}}\right)^{-1} Us_S \rangle \\
&= \langle s_T, \left(pI+\mathcal{L}(H_T)|_{\mathcal{K}_{H,T}}\right)^{-1} s_T \rangle \\
&=R_{H,T}(p).
\end{align*}

Hence, orthogonally equivalent triples have the same loop resolvent. The reverse implication follows from the preceding  argument: equality of the two resolvents gives the same visible eigenvalues and the same weights, hence the same triple up to orthogonal equivalence. This completes the proof.
\end{proof}

\subsection{Special case}

From Definition \ref{equivLaplacianmoment} and Theorem \ref{thm:heat_coefficients_loop_placement}, one observes that the equivariant heat character $\chi_t^{\widetilde{G}}(g)$ actually encodes much information regarding the underlying graph. In this subsection, we explore and seek solution to the question: \textit{what happens to the equivariant heat character $\chi_t^{\widetilde{G}}(g)$ when the non-empty loop set $S$ is exactly $V(G)$ for connected $G_S$, or in general, when $S$ is the union of connected components of a possibly disconnected $G_S$?} It turns out that in this case, for each $t >0,$ $\chi_t^{\widetilde{G}}(g)$ reduces to a quantity that depends only on $\sigma,$ not relying on any boundary information. Algebraically, the loop vector $s$ must necessarily lie in the kernel of $\mathcal{L}(G).$

\begin{theorem}
	\label{thm:character_special_case}
	Let $G_S$ be a (possibly disconnected) self-loop graph with loop vector $s.$ The following statements are equivalent:
	\begin{enumerate}[(i)]
		\item $S$ is a union of connected components of $G$;
		\item $\mathcal{L}(G_S)s=s$;
		\item $R_{G,S}(p)=\sigma/(p+1)$ for every $p>0$;
		\item $\chi_t^{\widetilde{G}}(g)=1+e^{-(2\sigma+1)t}-e^{-t}$ for every $t>0$.
	\end{enumerate}
\end{theorem}

\begin{proof}
	Observe that since  $s_i\in\{0,1\},$ we have $\mathcal{L}(G_S)s = \mathcal{L}(G)s+\mathrm{diag}(s)s = \mathcal{L}(G)s+s.$ Hence, $\mathcal{L}(G_S)s=s$ if and only if $\mathcal{L}(G)s=0$. For the ordinary graph Laplacian, $\mathcal{L}(G)x=0$ if and only if $x$ is constant on each connected component of $G$. Since $s$ is a $0$-$1$ vector, which means that every connected component of $G$ is either contained in $S$ or disjoint from $S$. Thus, $S$ is a union of connected components of $G$. This proves  $(i)\Leftrightarrow(ii)$.
		
	Assume $(ii)$. Then, $(pI+\mathcal{L}(G_S))s=(p+1)s.$ Since $p \neq -1,$ it follows that
	\begin{equation}\label{eq:R_GS_sigma}
		R_{G,S}(p) = s^\top(pI+\mathcal{L}(G_S))^{-1}s = \frac{s^\top s}{p+1} =	\frac{\sigma}{p+1}.
	\end{equation}
	This proves $(ii)\Rightarrow(iii)$.
	
	Conversely, assume $(iii)$ and consider \eqref{eq:RGSspecrep} in Theorem~\ref{thm:loop_resolvent_spectral_measure_and_moments}. By Theorem~\ref{thm:loop_resolvent_spectral_measure_and_moments}(i), both sides extend to rational functions of $p$. Since they agree for every $p>0$, they agree identically as rational functions. Since all coefficients $\|P_\lambda s\|^2$ are nonnegative and
	\[
	\sum_\lambda \|P_\lambda s\|^2=s^\top s=\sigma,
	\]
	the only possible nonzero spectral weight of $s$ is at $\lambda=1$.	Indeed, any term with $\lambda\neq1$ would produce a pole at $p=-\lambda\neq-1$, whereas the right-hand side has no such pole. Hence, $P_1s=s$, and therefore $\mathcal{L}(G_S)s=s$. This proves $(iii)\Rightarrow(ii)$.
	
	We now show $(iii) \Rightarrow (iv)$. Substituting \eqref{eq:R_GS_sigma} into \eqref{eq:equivariant_resolvent1} gives
	\begin{align*}
		\int_0^\infty e^{-pt}\chi_t^{\widetilde{G}}(g)\,dt
		&= \frac{d}{dp} \log\left(p+2\sigma-\frac{2\sigma}{p+1}\right) \\
		&= \frac{d}{dp} \log\left(\frac{p(p+2\sigma+1)}{p+1}\right) \\
		&= \frac{1}{p}+\frac{1}{p+2\sigma+1}-\frac{1}{p+1}.
	\end{align*}
	Taking inverse Laplace transforms, we obtain
	\begin{align*}
		\chi_t^{\widetilde{G}}(g)
		&= \mathcal{L}^{-1}_{p\to t}\left[\frac{1}{p}+\frac{1}{p+2\sigma+1}-\frac{1}{p+1}\right] \\
		&= \mathcal{L}^{-1}_{p\to t}\left[\int_0^\infty e^{-p\tau}\left(	1+e^{-(2\sigma+1)\tau}-e^{-\tau}\right)\,d\tau	\right] \\
		&= 1+e^{-(2\sigma+1)t}-e^{-t}.
	\end{align*}
	
	Finally, assume $(iv)$. Observe that the Taylor expansion of $\chi_t^{\widetilde{G}}(g)$ gives the third heat coefficient
	\[
	\Theta_3=(2\sigma+1)^3-1 = 8\sigma^3 + 12\sigma^2 + 6\sigma.
	\]
	On the other hand, by Theorem~\ref{thm:heat_coefficients_loop_placement}, we have $\Theta_3=8\sigma^3+12\sigma^2+6\sigma+6|\partial S|,$ which forces $|\partial S|=0$. Hence, no edge of $G$ joins $S$ to $V(G)\setminus S$, so $S$ is a union of connected components of $G$. This proves
	$(iv)\Rightarrow(i)$.
\end{proof}

After establishing the results above, let's denote this character as 
\[
\overline{\chi}_t^{\widetilde{G}}(g):=1+e^{-(2\sigma+1)t}-e^{-t}, \quad t>0,
\] to signify the special case.

\begin{corollary}
	If $G_S$ is connected and $\emptyset\neq S\neq V(G)$, then $\chi_t^{\widetilde{G}}(g)<\overline{\chi}_t^{\widetilde{G}}(g)$ for all sufficiently small $t>0$. More precisely,
	\begin{equation}\label{eq:chidiff1}
	\chi_t^{\widetilde{G}}(g)-\overline{\chi}_t^{\widetilde{G}}(g)
	= -|\partial S|t^3 + \frac{(16\sigma+16)|\partial S|+8\kappa(S)}{24}t^4 + O(t^5).
	\end{equation}
\end{corollary}

\begin{proof}
	By assumption, we have $|\partial S|>0.$
	Using the expansion
	\[
	\overline{\chi}_t^{\widetilde{G}}(g)
	=1+\sum_{k=1}^{\infty}\frac{(-t)^k}{k!}\left((2\sigma+1)^k-1\right)
	\]
	and comparing with the coefficients up in Theorem~\ref{thm:heat_coefficients_loop_placement}, we observe that the terms up to order two coincide and
	\[
	\Theta_3-\left((2\sigma+1)^3-1\right) = 6|\partial S|,
	\]
	\[
	\Theta_4-\left((2\sigma+1)^4-1\right) =	(16\sigma+16)|\partial S|+8\kappa(S),
	\]
	where \eqref{eq:chidiff1} follows.
\end{proof}

If $G_S$ is a full-loop connected self-loop graph, we may then use Theorem \ref{thm:character_special_case} to deduce its counterpart $\chi_t^{\widetilde{G}}(1)$ as given in the corollary below. The proof is omitted.

\begin{corollary}\label{fullloopchi(1)}
	Let $G$ be a connected simple graph of order $N$ and let $\widehat{G}$ be its full-loop graph. If $0=\theta_1<\theta_2\le \cdots\le \theta_N$ are the Laplacian eigenvalues of $G$,
	then 
	\[
	\chi_t^{\widetilde{G}}(1) = 1+e^{-(2N+1)t}+e^{-t} + 2\sum_{j=2}^{N}e^{-(\theta_j+1)t}.
	\]
\end{corollary}

Thus, as one observes \textit{a priori} that in the full-loop case any self-loop graph of order $N$ and $|S|=N$ share the same $\chi_t^{\widetilde{G}}(g);$ whilst $\chi_t^{\widetilde{G}}(1)$ is the invariant that contains the spectral information of the underlying graph.

\begin{example}	
	\begin{enumerate}[(i)]
		\item If $G$ is the complete graph $K_N$, then
		\[
		\chi_t^{\widetilde{K_N}}(1) = 1 + e^{-(2N+1)t}+e^{-t}+2(N-1)e^{-(N+1)t}.
		\]
		
		\item If $G$ is the complete bipartite graph $K_{a,b}$ with $N=a+b$, then
		\[
		\chi_t^{\widetilde{K_{a,b}}}(1) = 1 + e^{-(2N+1)t}+e^{-t} + 2e^{-(N+1)t} + 2(b-1)e^{-(a+1)t}	+ 2(a-1)e^{-(b+1)t}.
		\]
		
		\item If $G$ is the cycle graph $C_N$ with $N\ge 3$, then
		\[
		\chi_t^{\widetilde{G}}(1)
		= 1+e^{-(2N+1)t}+e^{-t}	+ 2\sum_{j=1}^{N-1}\exp\!\left(-\left(3-2\cos\left(\frac{2\pi j}{N}\right)\right)t\right).
		\]		
	\end{enumerate}
\end{example}

We end this section with two results which might be of independent interest, one on the join of full-loop graphs, and one on connected regular graph in relation to its line graph. 

\begin{proposition}
	\label{prop:character_join_full_loop}
	Let $G_1$ and $G_2$ be nonempty simple graphs of orders $N_1$ and $N_2$, respectively. Let $0=\alpha_1\leq \alpha_2\leq\cdots\leq\alpha_{N_1}$	
	and $0=\beta_1\leq \beta_2\leq\cdots\leq\beta_{N_2}$ be the Laplacian eigenvalues of $G_1$ and $G_2$, respectively. Let $N=N_1+N_2.$
	Then, for every $t>0$, $\chi_t^{\widetilde{\widehat{G_1}\vee\widehat{G_2}}}(g) = 1+e^{-(2N+1)t}-e^{-t},$ and
	\[
	\chi_t^{\widetilde{\widehat{G_1}\vee\widehat{G_2}}}(1) 
	= 1+e^{-(2N+1)t}+e^{-t} + 2e^{-(N+1)t} + 2\sum_{i=2}^{N_1}e^{-(\alpha_i+N_2+1)t} + 2\sum_{j=2}^{N_2}e^{-(\beta_j+N_1+1)t}.
	\]
\end{proposition}

\begin{proof}
	The Laplacian matrix of the join is 
	\[
	\mathcal{L}(G_1 \vee G_2)=
	\begin{pmatrix}
	\mathcal{L}(G_1) + N_2 I_{N_1} & -J_{N_1 \times N_2} \\
	-J_{N_2 \times N_1} & 	\mathcal{L}(G_2) + N_1 I_{N_2}
	\end{pmatrix}.
	\]
	Thus, the Laplacian spectrum of the join $G_1\vee G_2$ can be obtained to be
	\[
	\left\{
	0, N, \alpha_2+N_2,\ldots,\alpha_{N_1}+N_2,
	\beta_2+N_1,\ldots,\beta_{N_2}+N_1
	\right\}.
	\]
	Since every vertex of $\widehat{G_1}\vee\widehat{G_2}$ carries a loop, this graph is precisely the full-loop graph $\widehat{G_1\vee G_2}.$ Hence, the full-loop character formulae for $\chi_t^{\widetilde{\widehat{G_1}\vee\widehat{G_2}}}(g)$ and  $\chi_t^{\widetilde{\widehat{G_1}\vee\widehat{G_2}}}(1)$ follow from Theorem \ref{thm:character_special_case}(iv) and Corollary \ref{fullloopchi(1)}, respectively.
\end{proof}

The line graph of a self-loop graph has recently been studied by the author et. al. in \cite{akbari2024line}. For the readers' convenience, let us first recall the definition.

\begin{definition}\cite[Def. 2.4]{akbari2024line}\label{def:linegraph}
For any self-loop graph $G_S,$ the line graph $L(G_S)$ is a graph whose vertices are the edges of $G_S$, with two vertices in $L(G_S)$ adjacent whenever the corresponding edges in $G_S$ have exactly one vertex in common. Each loop attached at a vertex $v \in V(G_S)$ is the vertex with a loop in $L(G_S)$, and this vertex is adjacent to those vertices in $L(G_S)$ which correspond to the edges of $G_S$ incident with the vertex $v\in V(G_S)$.
\end{definition}
In particular, the adjacency matrix of $L(G_S)$ is given by 
\[
A(L(G_S))
=\begin{pmatrix}
	A(L(G)) & F^\top \\
	F  & I_\sigma
 \end{pmatrix},
\] where $F$ is the loop-edge adjacency matrix.  If $\sigma =n,$ then $F$ coincides with the incidence matrix $B(G)$ of $G.$ 

Definition \ref{def:linegraph} indicates that the line graph $L(G_S)$ is not a full-loop graph even if $G_S$ is full-loop, which means Theorem \ref{thm:character_special_case}(iv) cannot be applied directly for the line graph of $\widehat{G}$. 

Before proving the result, we shall first prove the Laplacian analog of \cite[Theorem 2.8]{akbari2024line} that expresses the relation between the Laplacian characteristic polynomial of $L(\widehat{G})$ and $G.$

\begin{lemma}
	Let $G$ be a connected $r$-regular simple graph of order $N$ and size $m=Nr/2, r\geq 2$. Let $H=L(\widehat{G})$ be the line graph of $\widehat{G}.$ Let $\mathcal L(H)$ and $\mathcal L(G)$ be the Laplacian matrices for $H$ and $G,$ respectively. Then,
	\[
	\Phi_{\mathcal L(H)}(x)	= (x-2r-2)^{m-N}(x-r-2)^N \Phi_{\mathcal L(G)}\left(\frac{x^2-(r+3)x+2}{x-r-2}\right).
	\]
\end{lemma}

\begin{proof}
	First, order the vertices of $H$ by the ordinary edges of $G$ followed by the loops of $\widehat{G}$. Since $B^\top B=A(L(G))+2I_m$, we obtain
	\[
	\mathcal L(H)
	= \begin{pmatrix}
		2rI_m-A(L(G))&-B^\top\\
		-B&(r+1)I_N
	\end{pmatrix} 	
	=\begin{pmatrix}
		(2r+2)I_m-B^\top B&-B^\top\\
		-B&(r+1)I_N
	\end{pmatrix}.
	\] 
	Consider the characteristic polynomials $\Phi_{\mathcal L(G)}(x)=\det(xI_N-\mathcal{L}(G))$ and $\Phi_{\mathcal L(H)}(x)=\det(xI_{m+N}-\mathcal{L}(H)).$ Then, applying the Schur complement and taking determinant gives 
	\begin{align*}
	\Phi_{\mathcal L(H)}(x)
	&= (x-r-1)^N \det\left((x-2r-2)I_m + B^\top B -	B^T \left(\frac{1}{x-r-1} I_N\right) B	\right) \\
	&= (x-r-1)^N \det\left((x-2r-2)I_m+	\frac{x-r-2}{x-r-1}B^\top B	\right) \\
	&= (x-2r-2)^{m-N}\det\left(\left[x^2-(r+3)x+2\right]I_N-(x-r-2)\mathcal{L}(G)\right) \\
	&= (x-2r-2)^{m-N}(x-r-2)^N 	\Phi_{\mathcal L(G)}\left(\frac{x^2-(r+3)x+2}{x-r-2}\right).
	\end{align*}	
	Here in the third equality, we use an identity $\det(aI_m + cB^\top B)=a^{m-N} \det(aI_N + cBB^\top)$ together with  $BB^\top=2rI_N-\mathcal{L}(G).$
\end{proof}

	It is an easy consequence that for each $\theta_j$ the corresponding eigenvalues of $\mathcal{L}(H)$ are the roots of $x^2-(\theta_j+r+3)x+(r+2)\theta_j+2=0.$
In particular, 
\[
\mathrm{Spec}(\mathcal{L}(H))
=\begin{bmatrix}
	\rho_1^+ & \rho_1^- & \rho_2^+ & \rho_2^- & \cdots &
	\rho_N^+ & \rho_N^- & 2r+2\\
	1 & 1 & 1 & 1 & \cdots & 1 & 1 & m-N
\end{bmatrix},
\]
where
\begin{equation} \label{eq:linegraph_eigenv1}
\rho_j^\pm = \frac{\theta_j+r+3 \pm \sqrt{(\theta_j-r-1)^2+4r}}{2}, \quad 1\leq j\leq N.
\end{equation}
Of course, to determine $\chi^{\widetilde{H}}_t(g),$ from the formula \eqref{eq:equivheatchar3}, we still need the counterpart from $M_{sym}.$ Observe that in this case, the loop vector of $H$ is given by $s_H=(0_m, \mathbf 1_N)^\top.$ Hence, the symmetric block associated with the A{\c c}{\i}kme{\c s}e
lift of $H$ is
\[
M_{sym}(H)
=\begin{pmatrix}
  2N & -\sqrt{2}s_H^\top\\
  -\sqrt{2}s_H & \mathcal L(H)
 \end{pmatrix}.
\]
Since $s_H=(0_m,\mathbf 1_N)^\top$ belongs to the two-dimensional $\mathcal L(H)$-invariant subspace associated with $\theta_1=0$, adjoining the central coordinate yields a three-dimensional invariant subspace of $M_{sym}(H)$. On this three-dimensional subspace, $M_{sym}(H)$ is represented by
\[
\begin{pmatrix}
	2N & 0 & -\sqrt{2N}\\
	0  & 2 & -\sqrt{2r}\\
	-\sqrt{2N} &-\sqrt{2r} & r+1
\end{pmatrix},
\]
whose eigenvalues can be determined as $0$ and $\eta^{\pm},$ where
\begin{equation} \label{eq:linegraph_eigenv2}
\eta^\pm = \frac{2N+r+3\pm\sqrt{(2N-r-1)^2+4r}}{2}.
\end{equation}

Since $s_H$ belongs to the two-dimensional $\mathcal L(H)$-invariant subspace associated with $\theta_1=0$, its orthogonal complement is also invariant under $\mathcal L(H)$. Moreover, every vector $w$ in this orthogonal complement satisfies $s_H^\top w=0$. Hence,
\[
M_{sym}(H)
\begin{pmatrix}
	0\\
	w
\end{pmatrix}
=
\begin{pmatrix}
	0\\
	\mathcal L(H)w
\end{pmatrix}.
\]
Therefore, on the complementary invariant subspace, $M_{sym}(H)$ and $\mathcal L(H)$ have the same eigenvalues with the same multiplicities, and these eigenvalues cancel in the equivariant heat character.

Combining all, we can now state the result on the equivariant heat character for the A\c{c}{\i}kme\c{s}e lift of the line graph of a full-loop $r$-regular graph :
\begin{proposition}
	\label{prop:line_graph_full_loop_character}
	Let $G$ be a connected $r$-regular simple graph of order $N$ and size $m=Nr/2, r\geq 2$, and let $\widehat{G}$ be its full-loop graph. Let $H=L(\widehat{G})$ be the line graph of $\widehat{G}.$ Then, for every $t>0$,
	\[
	\chi_t^{\widetilde{H}}(g) = 1 + e^{-\eta^-t}+e^{-\eta^+t} - e^{-\rho_1^-t} - e^{-\rho_1^+t},
	\] 
	where $\rho_1^\pm$ (i.e., $\theta_1=0$) and $\eta^{\pm}$ are given in \eqref{eq:linegraph_eigenv1} and \eqref{eq:linegraph_eigenv2}, respectively.
\end{proposition}

\subsection{Regularized Equivariant Heat Integral}

Throughout this section, we shall assume $G_S$ is pseudo-connected. In particular, $\mathcal{L}(G_S)$ is positive definite by Lemma \ref{lem:positive-definite}, and $\mathcal{L}(G_S)^{-1}$ is an ordinary inverse. The A\c{c}{\i}kme\c{s}e lift graph $\widetilde{G}$ is assumed to be connected, and $M_{sym}$ has exactly one zero eigenvalue. 

Observe that $\lim_{t \to \infty} \chi_t^{\widetilde{G}}(g)=1,$ indicating $\int_0^\infty \chi_t^{\widetilde{G}}(g) dt$ diverges. This motivates the following definition on the modified integral. 

\begin{definition}
	\label{def:regularized_eq_heat_integral}
	Let $\widetilde{G}$ be the A\c{c}{\i}kme\c{s}e lift of a self-loop graph $G_S.$ Let $\chi^{\widetilde{G}}_t(g)$ be the equivariant heat character of $g.$ Define the \textit{regularized equivariant heat integral} of $\widetilde{G}$ by
	\[
	\mathcal{I}_{eq}(\widetilde{G}) := \int_0^\infty \left( \chi_t^{\widetilde{G}}(g)-1 \right)dt.
	\]
\end{definition}

\begin{theorem}
	\label{thm:spectral_form_Ieq}
	Let $\nu_1,\dots,\nu_N$ be the positive eigenvalues of $M_{sym}$, and let $\lambda_1,\dots,\lambda_N$ be the eigenvalues of $\mathcal{L}(G_S)$. Then,
	\begin{equation}\label{eq:Ieq1}
	\mathcal{I}_{eq}(\widetilde{G})	= \sum_{j=1}^N\frac{1}{\nu_j} - \sum_{i=1}^N\frac{1}{\lambda_i}.
	\end{equation}
	Equivalently,
	\begin{equation}\label{eq:Ieq2}
	\mathcal{I}_{eq}(\widetilde{G})	= \mathrm{Tr}(M_{sym}^{\dagger}) - \mathrm{Tr}(\mathcal{L}(G_S)^{-1}),
	\end{equation}
	where $M_{sym}^{\dagger}$ denotes the Moore-Penrose generalized inverse\footnote{We refer the readers to \cite[pp. 453]{HornJohnson2013} for its definition and properties.} of $M_{sym}$.
\end{theorem}

\begin{proof}
	Since $M_{sym}$ has only a zero eigenvalue and $\mathcal{L}(G_S)$ is positive definite, we have respectively  $\mathrm{Tr}(e^{-tM_{sym}})=1+\sum_{j=1}^Ne^{-t\nu_j}$ and $\mathrm{Tr}(e^{-t\mathcal{L}(G_S)})=\sum_{i=1}^Ne^{-t\lambda_i}$. Using \eqref{eq:equivheatchar3}, we obtain
	\[
	\chi_t^{\widetilde{G}}(g)-1	= \sum_{j=1}^Ne^{-t\nu_j} - \sum_{i=1}^Ne^{-t\lambda_i}.
	\]
	Integrating termwise (using the identity $\int_0^\infty e^{-ta} dt = 1/a$ for $a>0$) gives the stated formula \eqref{eq:Ieq1}. For the equivalent trace formula, take an orthogonal spectral decomposition $M_{sym}=Q \; \mathrm{diag}(0,\nu_1,\ldots, \nu_N) \; Q^\top.$ Then, its Moore-Penrose generalized inverse is
	\[
	M^{\dagger}_{sym}=Q \; \mathrm{diag}\left(0,\frac{1}{\nu_1},\ldots, \frac{1}{\nu_N}\right) \; Q^\top.
	\]

	Similarly, with $\mathcal{L}(G_S)=P \;\mathrm{diag}(\lambda_1,\ldots, \lambda_N) \; P^\top,$ its inverse is 
	\[
	\mathcal{L}(G_S)^{-1}=P\;\mathrm{diag}\left(\frac{1}{\lambda_1},\ldots, \frac{1}{\lambda_N} \right) \; P^\top.
	\]
	By taking the traces for both, \eqref{eq:Ieq2} follows immediately.
\end{proof}

\begin{theorem} \label{thm:resolvent_expression_Ieq}
	Let $\widetilde{G}$ be the A\c{c}{\i}kme\c{s}e lift of a self-loop graph $G_S.$ Then, 
	\begin{equation}\label{eq:Ieq3}
	\mathcal{I}_{eq}(\widetilde{G})	= -\frac{2s^\top\mathcal{L}(G_S)^{-3}s}{1+2s^\top\mathcal{L}(G_S)^{-2}s}.
	\end{equation}
\end{theorem}

\begin{proof}	
	Expand $R_{G,S}(p)$ around $p=0,$ by a computation and Corollary \ref{cor:Linv},
	\begin{equation}\label{eq:RGSexpansion2}
	p+2\sigma-2R_{G,S}(p) = p \left(1+2s^\top\mathcal{L}(G_S)^{-2}s	- 2ps^\top\mathcal{L}(G_S)^{-3}s + O(p^2) \right).
	\end{equation}
	For simplicity, let $H(p)=1+2s^\top\mathcal{L}(G_S)^{-2}s - 2ps^\top\mathcal{L}(G_S)^{-3}s + O(p^2)$ and $F(p)$ be the determinant fraction in \eqref{eq:equivariant_resolvent2}. Then, \eqref{eq:RGSexpansion2} is simply $F(p)=pH(p).$ First, if $\{\nu_j\}$ and $\{\lambda_j\}$ are the spectrum of $M_{sym}$ and $\mathcal{L}(G_S),$ respectively, then observe that 
	\[
	F(p) = p\frac{\prod_{j=1}^{N}(p+\nu_j)}{\prod_{j=1}^{N}(p+\lambda_j)}.
	\] Thus,
	\begin{equation}\label{eq:Fpderiv1}
	\frac{F'(p)}{F(p)} = \frac{1}{p}+ \sum^N_{j=1} \frac{1}{p+\nu_j}-\sum^N_{j=1} \frac{1}{p+\lambda_j}.
	\end{equation} 
	On the other hand, from $\log(F(p))= \log(p)+\log(H(p)),$ we also have
	\begin{equation}\label{eq:Fpderiv2}
	\frac{F'(p)}{F(p)} = \frac{1}{p} + \frac{H'(p)}{H(p)}.
	\end{equation} 
	Thus, by Theorem \ref{thm:spectral_form_Ieq} , together with Equations \eqref{eq:Fpderiv1} and \eqref{eq:Fpderiv2}, we obtain
	\begin{align*}
	\mathcal{I}_{eq}(\widetilde{G})
	&= \lim_{p\to 0} \left(\frac{F'(p)}{F(p)} -\frac{1}{p} \right) \\
	&= \lim_{p\to 0} \left(\frac{H'(p)}{H(p)}\right) \\
	&= \lim_{p\to 0} \left(\frac{-2s^\top \mathcal{L}(G_S)^{-3}s +O(p)}{1+2s^\top\mathcal{L}(G_S)^{-2}s - 2ps^\top\mathcal{L}(G_S)^{-3}s + O(p^2)}\right) \\
	&=  - \frac{2s^\top\mathcal{L}(G_S)^{-3}s}{1+2s^\top\mathcal{L}(G_S)^{-2}s}. \qedhere
	\end{align*}
\end{proof}

\begin{corollary}	\label{thm:Ieq_negative_bound}
	Let $\widetilde{G}$ be the A\c{c}{\i}kme\c{s}e lift of a self-loop graph $G_S.$  Then, the following bound holds:
	\[
	-\frac{1}{\lambda_{\min}}<\mathcal{I}_{eq}(\widetilde{G})<0.
	\]
\end{corollary}

\begin{proof}
	Since $\mathcal{L}(G_S)$ is positive definite and $s\neq 0$, the right-hand side of \eqref{eq:Ieq3} is strictly negative.
	Let $\lambda_{\min}:=\lambda_{\min}(\mathcal{L}(G_S)).$ Since $\mathcal{L}(G_S)^{-1}\le \lambda_{\min}^{-1}I$, we have
	$s^\top\mathcal{L}(G_S)^{-3}s
	\le \lambda_{\min}^{-1}s^\top\mathcal{L}(G_S)^{-2}s$. Hence, by Theorem~\ref{thm:resolvent_expression_Ieq}, we obtain
	\[
	-\mathcal{I}_{eq}(\widetilde{G})
	\le\frac{1}{1+2s^\top\mathcal{L}(G_S)^{-2}s} \cdot\frac{2s^\top\mathcal{L}(G_S)^{-2}s}{\lambda_{\min}}<\frac{1}{\lambda_{\min}}. 
	\]
\end{proof}

\begin{example}
	\label{ex:petersen_full_loop}
	Let $\mathrm{P}_{10}$ be the Petersen graph. If $\widehat{\mathrm{P}_{10}}$ is its full-loop graph, using the full-loop formula one computes that
	\[
	\chi_t^{\widetilde{\widehat{\mathrm{P}_{10}}}}(g) = 1+e^{-21t}-e^{-t} \quad \text{ and } \quad
	\mathcal I_{eq} \left(\widetilde{\widehat{\mathrm{P}_{10}}} \right)	=
	-\frac{20}{21}.
	\]
	On the other hand, suppose $S$ is the vertex set of a five-cycle in $\mathrm{P}_{10}$ and consider $(\mathrm{P}_{10})_S$ (so $N=10$ and	$\sigma=5$). With respect to the partition
	\[
	V(\mathrm{P}_{10})=S\sqcup\bigl(V(\mathrm{P}_{10})\setminus S\bigr),
	\]
	the quotient matrices of $\mathcal L((\mathrm{P}_{10})_S)$ and $M_{sym}$ are, respectively,
	\[
	Q_{-}
	=\begin{pmatrix}
		2&-1\\
		-1&1
	\end{pmatrix}\quad  \text{ and }
	\quad 
	Q_{+}
	=\begin{pmatrix}
		10&-5\sqrt2&0\\
		-\sqrt2&2&-1\\
		0&-1&1
	\end{pmatrix},
	\]
	with eigenvalues $\lambda_{\pm}=(3\pm\sqrt5)/2$ for $Q_{-}$ and $0,\eta_{\pm}=(13\pm\sqrt{85})/2$ for $Q_{+}.$ On the complementary invariant subspaces, \(M_{sym}\) and \(\mathcal L((\mathrm P_{10})_S)\) have the same eigenvalues, which therefore cancel in the equivariant heat character.	It follows that
	\[
	\chi_t^{\widetilde{(\mathrm{P}_{10})_S}}(g) = 1 + e^{-\eta_+t} + e^{-\eta_-t} - e^{-\lambda_+t} - e^{-\lambda_-t},
	\]
	\[
	\mathcal I_{eq}\left(\widetilde{(\mathrm{P}_{10})_S}\right)
	= \frac{1}{\eta_+}	+ \frac{1}{\eta_-} - \frac{1}{\lambda_+} - \frac{1}{\lambda_-}
	=\frac{13}{21}-3 = -\frac{50}{21}.
	\]
	Thus, 
	\[
	\mathcal I_{eq}	\left(\widetilde{\widehat{\mathrm{P}_{10}}}\right)	= -\frac{20}{21} \neq -\frac{50}{21}
	=\mathcal I_{eq}\left(\widetilde{(\mathrm{P}_{10})_S}\right).
	\]
	This indicates that even for the same underlying Petersen graph, the full-loop and $5$-loop cases yield distinct regularized equivariant heat integrals.
\end{example}

	\vspace{0.5cm}
	\textbf{Acknowledgment.}
	
	Johnny Lim acknowledges the support from the Ministry of Higher Education Malaysia for Fundamental Research Grant Scheme with Project Code: 
	FRGS/1/2025/STG06/USM/02/1.

	\vspace{0.5cm}
	\textbf{Conflicts of interest.} The author declares no conflict of interest. \\
	
	\textbf{Data Availability.} No data is required for this research.
	
	\bibliography{bibliography}{}

\providecommand{\bysame}{\leavevmode\hbox to3em{\hrulefill}\thinspace}
\providecommand{\MR}{\relax\ifhmode\unskip\space\fi MR }
\providecommand{\MRhref}[2]{%
  \href{http://www.ams.org/mathscinet-getitem?mr=#1}{#2}
}
\providecommand{\href}[2]{#2}
\begin{thebibliography}{10}

\bibitem{acikmese2015spectrum}
B.~A\c{c}{\i}kme\c{s}e, \emph{Spectrum of {L}aplacians for {G}raphs with
  {S}elf-{L}oops}, arXiv preprint arXiv:1505.08133 (2015).

\bibitem{akbari2023selfloop}
S.~Akbari, H.~Al~Menderj, M.~H. Ang, J.~Lim, and Z.~C. Ng, \emph{Some {R}esults
  on {S}pectrum and {E}nergy of {G}raphs with {L}oops}, Bull. Malays. Math.
  Sci. Soc. \textbf{46} (2023), no.~3, Paper No. 94, 18. \MR{4567384}

\bibitem{akbari2024line}
S.~Akbari, I.~M. Jovanovi\'c, and J.~Lim, \emph{Line graphs and
  {N}ordhaus-{G}addum-type bounds for self-loop graphs}, Bull. Malays. Math.
  Sci. Soc. \textbf{47} (2024), no.~4, Paper No. 117, 22. \MR{4751718}

\bibitem{anchan20232}
D.~V. Anchan, S.~D’Souza, H.~J. Gowtham, and P.~G. Bhat, \emph{Laplacian
  {E}nergy of a {G}raph with {S}elf-{L}oops}, MATCH Commun. Math. Comput. Chem.
  \textbf{90} (2023), no.~1, 247--258.

\bibitem{biggs1993algebraic}
N.~Biggs, \emph{Algebraic {G}raph {T}heory}, no.~67, Cambridge University
  Press, 1993.

\bibitem{BrouwerHaemers}
A.~E. Brouwer and W.~H. Haemers, \emph{Spectra of {G}raphs}, first ed.,
  Universitext, Springer New York, NY, 2011.

\bibitem{chung1997spectral}
F.~R.K. Chung, \emph{Spectral {G}raph {T}heory}, vol.~92, American Mathematical
  Soc., 1997.

\bibitem{cvetkovic1995spectra}
D.~M. Cvetkovi\'{c}, M.~Doob, and H.~Sachs, \emph{Spectra of {G}raphs}, third
  ed., Johann Ambrosius Barth, Heidelberg, 1995, Theory and applications.
  \MR{1324340}

\bibitem{cvetkovic2010intro}
D.~M. Cvetkovi\'{c}, P.~Rowlinson, and S.~Simi\'{c}, \emph{An {I}ntroduction to
  the {T}heory of {G}raph {S}pectra}, London Mathematical Society Student
  Texts, vol.~75, Cambridge University Press, Cambridge, 2010. \MR{2571608}

\bibitem{kinkar2014laplacian}
K.~Ch. Das and S.~A. Mojallal, \emph{On {L}aplacian energy of graphs}, Discrete
  Math. \textbf{325} (2014), 52--64. \MR{3181233}

\bibitem{demuth2005determining}
M.~Demuth and M.~Krishna, \emph{Determining {S}pectra in {Q}uantum {T}heory},
  Progress in Mathematical Physics, vol.~44, Birkh{\"a}user, 2005.

\bibitem{dummit2004abstract}
D.~S. Dummit and R.~M. Foote, \emph{Abstract {A}lgebra}, 3rd ed., John Wiley \&
  Sons, 2004.

\bibitem{engel2000oneparameter}
K.-J. Engel and R.~Nagel, \emph{One-{P}arameter {S}emigroups for {L}inear
  {E}volution {E}quations}, Graduate Texts in Mathematics, vol. 194, Springer,
  New York, 2000.

\bibitem{godsil2001algebraic}
C.~Godsil and G.~F. Royle, \emph{Algebraic {G}raph {T}heory}, Graduate Texts in
  Mathematics, vol. 207, Springer New York, 2001.

\bibitem{gutman2021energy}
I.~Gutman, I.~Red{\v{z}}epovi{\'c}, B.~Furtula, and A.~Sahal, \emph{Energy of
  {G}raphs with {S}elf-{L}oops}, MATCH Commun. Math. Comput. Chem. \textbf{87}
  (2021), 645--652.

\bibitem{higham2008functions}
N.~J. Higham, \emph{Functions of {M}atrices: {T}heory and {C}omputation},
  Society for Industrial and Applied Mathematics, Philadelphia, 2008.

\bibitem{HornJohnson2013}
R.~A. Horn and C.~R. Johnson, \emph{Matrix analysis}, second ed., Cambridge
  University Press, Cambridge, 2013. \MR{2978290}

\bibitem{jovanovic2023}
I.~Jovanovi{\'c}, E.~Zogi{\'c}, and E.~Glogi{\'c}, \emph{On the conjecture
  related to the energy of graphs with self-loops}, MATCH Commun. Math. Comput.
  Chem. \textbf{89} (2023), 479--488.

\bibitem{lim2025closedwalk}
J.~Lim, \emph{Closed walks of low dimension and twisted moments on self-loop
  graphs}, Bull. Malays. Math. Sci. Soc. \textbf{48} (2025), no.~194.

\bibitem{MajKloGut2009}
S.~Majstorovi\'{c}, A.~Klobu\v{c}ar, and I.~Gutman, \emph{Selected topics from
  the theory of graph energy: hypoenergetic graphs}, Zb. Rad. (Beogr.)
  \textbf{13(21)} (2009), 65--105. \MR{2543254}

\bibitem{marshall2011inequalities}
A.~W. Marshall, I.~Olkin, and B.~C. Arnold, \emph{Inequalities: {T}heory of
  {M}ajorization and {I}ts {A}pplications}, 2 ed., Springer Science \& Business
  Media, 2011.

\bibitem{serre1977linear}
J.-P. Serre, \emph{Linear representations of finite groups}, vol.~42, Springer,
  1977.

\bibitem{stewart2022galois}
I~Stewart, \emph{Galois {T}heory}, 5th ed., CRC Press, 2022.

\bibitem{teschl2000jacobi}
G.~Teschl, \emph{Jacobi {O}perators and {C}ompletely {I}ntegrable {N}onlinear
  {L}attices}, Mathematical Surveys and Monographs, vol.~72, American
  Mathematical Society, Providence, RI, 2000.

\bibitem{theobald1975inequality}
C.~M. Theobald, \emph{An inequality for the trace of the product of two
  symmetric matrices}, Mathematical Proceedings of the Cambridge Philosophical
  Society \textbf{77} (1975), no.~2, 265--267.

\end{thebibliography}
	\bibliographystyle{amsplain}
	

	\end{document}